\documentclass[a4paper,12pt]{amsart}
\usepackage{amsmath, amsfonts,amsthm,times,graphics}
\usepackage{mathrsfs}
\usepackage[usenames]{color}
\usepackage[active]{srcltx}
 \makeatletter
\renewcommand*\subjclass[2][2000]{%
  \def\@subjclass{#2}%
  \@ifundefined{subjclassname@#1}{%
    \ClassWarning{\@classname}{Unknown edition (#1) of Mathematics
      Subject Classification; using '1991'.}%
  }{%
    \@xp\let\@xp\subjclassname\csname subjclassname@#1\endcsname
  }%
}
 \makeatother

\newtheorem*{ThmA}{Theorem A}
\newtheorem*{ThmB}{Theorem B}
\newtheorem*{ThmC}{Theorem C}
\newtheorem*{ThmD}{Theorem D}
\newtheorem*{ThmE}{Theorem E}

\newtheorem*{LemF}{Lemma F}
\newtheorem*{LemG}{Lemma G}

\newtheorem*{CorA}{Corollary A}

\newtheorem{Thm}{Theorem}[section]
\newtheorem{Cor}[Thm]{Corollary}
\newtheorem{Lem}[Thm]{Lemma}

\theoremstyle{definition}

\theoremstyle{remark}
\newtheorem{Rem}[Thm]{\upshape\bfseries Remark}

\newtheorem{Prob}[Thm]{Problem}

\numberwithin{equation}{section}

\newcommand{\ee}{\mathrm{e}}

\theoremstyle{definition}

\def\be{\begin{equation}}
\def\ee{\end{equation}}

\newcommand{\ben}{\begin{enumerate}}
\newcommand{\een}{\end{enumerate}}

\newcommand{\br}{\begin{rem}}
\newcommand{\er}{\end{rem}}
\newcommand{\brs}{\begin{rems}}
\newcommand{\ers}{\end{rems}}
\newcommand{\bo}{\begin{obser}}
\newcommand{\eo}{\end{obser}}
\newcommand{\bos}{\begin{obsers}}
\newcommand{\eos}{\end{obsers}}
\newcommand{\bpf}{\begin{pf}}
\newcommand{\epf}{\end{pf}}
\newcommand{\ba}{\begin{array}}
\newcommand{\ea}{\end{array}}
\newcommand{\beq}{\begin{eqnarray}}
\newcommand{\beqq}{\begin{eqnarray*}}
\newcommand{\eeq}{\end{eqnarray}}
\newcommand{\eeqq}{\end{eqnarray*}}

\numberwithin{equation}{section}

\newcounter{minutes}
\divide\time by 60
\newcounter{hours}
\multiply\time by 60 \addtocounter{minutes}{-\time}
\begin{document}
\title{Equivalent norms and Hardy-Littlewood type Theorems, and their applications}

\author[Shaolin Chen and Hidetaka Hamada]{Shaolin Chen and Hidetaka Hamada}

\address{S. L.  Chen, College of Mathematics and
Statistics, Hengyang Normal University, Hengyang, Hunan 421002,
People's Republic of China; Hunan Provincial Key Laboratory of Intelligent Information Processing and Application,  421002,
People's Republic of China.} \email{mathechen@126.com}

\address{H. Hamada, Faculty of Science and Engineering, Kyushu Sangyo University,
3-1 Matsukadai 2-Chome, Higashi-ku, Fukuoka 813-8503, Japan.}
\email{ h.hamada@ip.kyusan-u.ac.jp}


\maketitle

\def\thefootnote{}
\makeatletter\def\thefootnote{\@arabic\c@footnote}\makeatother

\begin{abstract}
The main purpose of this paper is to develop some methods to investigate equivalent norms and Hardy-Littlewood type Theorems on Lipschitz type spaces of analytic functions and
complex-valued
harmonic functions. Initially, some characterizations of equivalent norms on Lipschitz type spaces  of analytic functions and complex-valued
harmonic functions will be given. In particular, we give an answer to an open problem posed by Dyakonov in ({\it  Math. Z.} 249(2005), 597--611). Furthermore, some Hardy-Littlewood type Theorems of complex-valued
harmonic functions are established. The obtained results  improve and extend the main results in ({\it Acta Math.} 178(1997), 143--167).
 Additionally, we apply the equivalent norms and Hardy-Littlewood type Theorems to study composition operators between Lipschitz type spaces. 
\end{abstract}

\maketitle
\tableofcontents

\section{Introduction}\label{sec1}
Let $\mathbb{C}$ be the complex plane. For $a\in\mathbb{C}$ and $\rho>0$, let $\mathbb{D}(a,\rho)=\{z:~|z-a|<\rho\}$. In particular, we use $\mathbb{D}_{\rho}$ to
denote the disk $\mathbb{D}(0,\rho)$ and $\mathbb{D}$ to denote the unit disk $\mathbb{D}_{1}$. Moreover, let $\mathbb{T}:=\partial\mathbb{D}$ be the
unit circle.
A twice differentiable complex-valued function $f=u+iv$  is harmonic defined in a domain $\Omega\subset\mathbb{C}$ if the real-valued
 functions $u$ and $v$ satisfy Laplace's equations $\Delta u=\Delta v=0$, where
$$\Delta:=\frac{\partial^{2}}{\partial x^{2}}+\frac{\partial^{2}}{\partial y^{2}}.$$ For some more details of complex-valued harmonic functions, see \cite{Du}.
 It is well-known that
every complex-valued harmonic function $f$ defined in a simply
connected domain $\Omega$ admits a decomposition $f = h
+ \overline{g}$, where $h$ and $g$ are analytic in $\Omega$. This decomposition is unique up to an additive constant.
Let us recall that the Poisson integral, $P[\varphi]$, of a function $\varphi\in L^{1}(\mathbb{T})$ is defined by
$$P[\varphi](z)=\frac{1}{2\pi}\int_{0}^{2\pi}\varphi(e^{i\tau})\mathbf{P}(z,e^{i\tau})d\tau$$ or
$$P[\varphi](z)=\frac{1}{2\pi}\int_{-\pi}^{\pi}\varphi(e^{i\tau})\mathbf{P}(z,e^{i\tau})d\tau,$$
where $\mathbf{P}(z,e^{i\tau})=\frac{1-|z|^{2}}{|e^{i\tau}-z|^{2}}$
is the Poisson kernel. In the following, we use $\mathscr{H}(\mathbb{D})$ to denote the set of all complex-valued harmonic functions of $\mathbb{D}$ into $\mathbb{C}$.

For $z=x+iy\in\mathbb{C}$, the complex formal differential operators
are defined by $\partial/\partial z=1/2(\partial/\partial x-i\partial/\partial y)$ and $\partial/\partial \overline{z}=1/2(\partial/\partial x+i\partial/\partial y)$, where $x,y\in\mathbb{R}$.
For a differentiable function $f$ on a domain $\Omega\subset \mathbb{C}$,
let
$$\mathscr{M}_{f}(z):=|f_{z}(z)|+|f_{\overline{z}}(z)|,
\quad z\in \Omega.
$$
For
$\theta\in[0,2\pi]$, the  directional derivative of  a
complex-valued harmonic function $f$ at $z\in\mathbb{D}$ is defined
by
$$\partial_{\theta}f(z)=\lim_{\rho\rightarrow0^{+}}\frac{f(z+\rho e^{i\theta})-f(z)}{\rho}=f_{z}(z)e^{i\theta}
+f_{\overline{z}}(z)e^{-i\theta},$$ where $f_{z}:=\partial
f/\partial z,$ $f_{\overline{z}}:=\partial f/\partial \overline{z}$
and $\rho$ is a positive real number such that $z+\rho
e^{i\theta}\in\mathbb{D}$. Then
$$\mathscr{M}_{f}(z)=\max\{|\partial_{\theta}f(z)|:\; \theta\in[0,2\pi]\}.
$$

A mapping $f:~\Omega\rightarrow\mathbb{C}$ is said to be absolutely
continuous on lines, $ACL$ in brief, in the domain $\Omega$ if for
every closed rectangle $R\subset\Omega$ with sides parallel to the
axes $x$ and $y$, $f$ is absolutely continuous on almost every
horizontal line and almost every vertical line in $R$. Such a
mapping has, of course,  partial derivatives $f_{x}$ and $f_{y}$
a.e. in $\Omega$. Moreover, we say $f\in ACL^{2}$ if $f\in ACL$ and
its partial derivatives are locally $L^{2}$ integrable in $\Omega$.
A sense-preserving and continuous mapping $f$ of  $\mathbb{D}$ into
$\mathbb{C}$ is called a  $K$-quasiregular mapping if

\begin{enumerate}
\item  $f$ is $ACL^{2}$ in  $\mathbb{D}$ and $J_{f}>0$ a.e. in   $\mathbb{D}$, where $J_{f}=|f_{z}|^{2}-|f_{\overline{z}}|^{2}$ is the Jacobian of $f$;

\item there is a constant $K\geq1$
 such that
$$\mathscr{M}_{f}^{2}\leq KJ_{f}~\mbox{ a.e. in   $\mathbb{D}$}.$$
\end{enumerate}

Throughout of this paper,
 we use the symbol $M$ to denote the various positive
constants, whose value may change from one occurrence to another.

\section{Preliminaries and main results}\label{sec2-1}

A continuous increasing function $\omega:[0,\infty)\rightarrow[0,\infty)$ with $\omega(0)=0$ is called a majorant if
$\omega(t)/t$ is non-increasing for $t\in(0,\infty)$ (see \cite{Dy1,Dy2,P}). For $\delta_{0}>0$ and $0<\delta<\delta_{0}$, we consider the
following conditions on a majorant $\omega$:
\be\label{eq2x}
\int_{0}^{\delta}\frac{\omega(t)}{t}\,dt\leq\, M\omega(\delta)
\ee
and
\be\label{eq3x}
\delta\int_{\delta}^{\infty}\frac{\omega(t)}{t^{2}}\,dt\leq\, M
\omega(\delta),
\ee
where $M$ denotes a positive constant. A majorant $\omega$ is henceforth  called fast (resp. slow) if condition (\ref{eq2x}) (resp. (\ref{eq3x}) ) is fulfilled.
In particular, a majorant $\omega$ is said to be  regular if it satisfies the
conditions (\ref{eq2x}) and (\ref{eq3x}) (see \cite{Dy1,Dy2}).

Given a majorant $\omega$ and a subset  $\Omega$ of
$\mathbb{C}$, a function $f$ of  $\Omega$ into $\mathbb{C}$ is
said to belong to the  Lipschitz  space
$\Lambda_{\omega}(\Omega)$
 if there is a positive constant $M$ such that
 \be\label{rrt-1}|f(z_{1})-f(z_{2})| \leq\,M\omega\left(|z_{1}-z_{2}|\right), \quad
z_{1},z_{2} \in \Omega.\ee
Furthermore, let
$$\|f\|_{\Lambda_{\omega}(\Omega),s}:=\sup_{z_{1},z_{2}\in\Omega,z_{1}\neq\,z_{2}}\frac{|f(z_{1})-f(z_{2})|}{\omega(|z_{1}-z_{2}|)}<\infty.$$
Note that if $\Omega$ is a proper subdomain of $\mathbb{C}$ and
$f\in \Lambda_{\omega}(\Omega)$, then $f$ is continuous on
$\overline{\Omega}$ and (\ref{rrt-1}) holds for $z,w \in
\overline{\Omega}$ (see \cite{Dy2}).
Furthermore,  we use  $\Lambda_{\omega,p}(\mathbb{D})$ to denote the class of all Borel functions
$f$ of $\mathbb{D}$ into $\mathbb{C}$ such that, for $z_{1},~z_{2}\in\mathbb{D}$,
$$\mathcal{L}_{p}[f](z_{1},z_{2})\leq M\omega(|z_{1}-z_{2}|),$$
where $M$ is a positive constant and
$$\mathcal{L}_{p}[f](z_{1},z_{2})=
\begin{cases}
\displaystyle\left(\int_{0}^{2\pi}|f(e^{i\eta}z_{1})-f(e^{i\eta}z_{2})|^{p}d\eta\right)^{\frac{1}{p}}
& \mbox{if } p\in(0,\infty),\\
\displaystyle|f(z_{1})-f(z_{2})| &\mbox{if } p=\infty.
\end{cases}
$$
The Lipschitz constant of $f\in\Lambda_{\omega,p}(\mathbb{D})$ is defined as follows $$\|f\|_{\Lambda_{\omega,p}(\Omega),s}:=\sup_{z_{1},z_{2}\in\Omega,z_{1}\neq\,z_{2}}\frac{\mathcal{L}_{p}[f](z_{1},z_{2})}{\omega(|z_{1}-z_{2}|)}<\infty.$$
 Obviously, $\|f\|_{\Lambda_{\omega,\infty}(\mathbb{D}),s}=\|f\|_{\Lambda_{\omega}(\mathbb{D}),s}$ and $\Lambda_{\omega,\infty}(\mathbb{D})=\Lambda_{\omega}(\mathbb{D})$.
 Moreover,  we define the space $\Lambda_{\omega,p}(\mathbb{T})$ consisting of those $f\in L^{p}(\mathbb{T})$
for which
$$\mathcal{L}_{p}[f](\xi_{1},\xi_{2})\leq M\omega(|\xi_{1}-\xi_{2}|),~\xi_{1},~\xi_{2}\in\mathbb{T},$$
where $M>0$ is a constant and
$$\mathcal{L}_{p}[f](\xi_{1},\xi_{2})=
\begin{cases}
\displaystyle\left(\int_{0}^{2\pi}|f(e^{i\eta}\xi_{1})-f(e^{i\eta}\xi_{2})|^{p}d\eta\right)^{\frac{1}{p}}
& \mbox{if } p\in(0,\infty),\\
\displaystyle|f(\xi_{1})-f(\xi_{2})| &\mbox{if } p=\infty.
\end{cases}
$$
  In particular, we say that
a function $f$ belongs to the  local Lipschitz space
$\mbox{loc}\Lambda_{\omega}(\Omega)$ if (\ref{rrt-1})
holds, with a fixed
positive constant $M$, whenever $z\in \Omega$ and
$|z-w|<\frac{1}{2}d_{\Omega}(z)$ (cf. \cite{Dy2,GM,L}), where $d_{\Omega}(z)$ is the
Euclidean distance between $z$ and the boundary of $\Omega$. Moreover, $\Omega$ is called a  $\Lambda_{\omega}$-extension domain if
$\Lambda_{\omega}(\Omega)=\mbox{loc}\Lambda_{\omega}(\Omega).$
On the geometric
characterization of $\Lambda_{\omega}$-extension domains, see
 \cite{GM}. In \cite{L}, Lappalainen
generalized the characterization of \cite{GM}, and proved that $\Omega$ is a
$\Lambda_{\omega}$-extension domain if and only if each pair of points
$z_{1},z_{2}\in \Omega$ can be joined by a rectifiable curve $\gamma\subset \Omega$
satisfying
\[
\int_{\gamma}\frac{\omega(d_{\Omega}(\zeta))}{d_{\Omega}(\zeta)}\,ds(\zeta)
\leq M\omega(|z_{1}-z_{2}|)
\]
with some fixed positive constant
$M$, where $ds$ stands for the arc length measure on
$\gamma$.  Furthermore, Lappalainen \cite[Theorem 4.12]{L} proved
that $\Lambda_{\omega}$-extension domains  exist only for fast majorants
$\omega$. In particular, $\mathbb{D}$ is a $\Lambda_{\omega}$-extension domain of $\mathbb{C}$ for fast majorant $\omega$ (see \cite{Dy2}).

\subsection{Equivalent norms  and Hardy-Littlewood type theorems on Lipschitz type spaces}

 Let $\mathscr{A}(\mathbb{D})$ denote the disk algebra, i.e., the class of analytic functions in $\mathbb{D}$ that are continuous up to the boundary.
In \cite{Dy1}, Dyakonov  gave some characterizations of  functions
 $f\in\mathscr{A}(\mathbb{D})\cap\Lambda_{\omega,\infty}(\mathbb{D})$ in terms of their equivalent norms (or their moduli). Let's recall the main results of \cite{Dy1}
as follows. 

\begin{ThmA}\label{Dya-1}{\rm(see \cite[Theorem 2]{Dy1};
cf. \cite[p. 598]{Dy3} \mbox{or} \cite[Theorem A]{P})}
Let $\omega$ be a fast majorant.
Then
$f\in\mathscr{A}(\mathbb{D})\cap\Lambda_{\omega,\infty}(\overline{\mathbb{D}})$ if and only if $|f|\in\mathscr{A}(\mathbb{D})\cap\Lambda_{\omega,\infty}(\overline{\mathbb{D}})$.
\end{ThmA}

In particular, if we take $\omega=\omega_{\alpha}$ in Theorem A, where $\alpha\in(0,1]$ is a constant and $\omega_{\alpha}(t)=t^{\alpha}$ for $t\geq0$, then
we have
\be\label{NB-Lip-1}|f|\in\Lambda_{\omega_{\alpha},\infty}(\mathbb{D})\Leftrightarrow f\in\Lambda_{\omega_{\alpha},\infty}(\mathbb{D}).\ee
Furthermore, in \cite{Dy3}, Dyakonov  posed an open problem on the extension of (\ref{NB-Lip-1}). Let us recall it as follows.

\begin{Prob}\label{Problem-1}
We do not know how to extend (\ref{NB-Lip-1}) to the whole $\Lambda_{\omega_{\alpha},\infty}(\mathbb{D})$-scale with $\alpha\in(0,\infty)$ (higher order $\Lambda_{\omega_{\alpha},\infty}(\mathbb{D})$-spaces
being defined in terms of higher order derivatives); to find the ``right" extension of (\ref{NB-Lip-1})  is an open problem that puzzles us (see \cite[page 606 and lines 5-7]{Dy3}).
\end{Prob}

We give an answer to Problem \ref{Problem-1} as follows.

\begin{Thm}\label{Open-1} It is impossible to  extend {\rm(\ref{NB-Lip-1})} to the whole $\Lambda_{\omega_{\alpha},\infty}(\mathbb{D})$-scale with $\alpha\in(0,\infty)$ unless the analytic function $f$
is a constant.
\end{Thm}

In \cite{Dy3}, Dyakonov  also proved that the following original Hardy-Littlewood theorem follows from (\ref{NB-Lip-1}).

\begin{CorA}\label{CorA-HL}{\rm (Hardy-Littlewood's Theorem)}
Suppose $u$ is a real-valued harmonic function in $\mathbb{D}$ with
 $u\in\Lambda_{\omega_{\alpha},\infty}(\mathbb{D})$, where $\alpha\in(0,1]$ and $\omega_{\alpha}(t)=t^{\alpha}$ for $t\geq0$.
  Let $v$ be a harmonic conjugate of $u$ with $v(0)=0$. Then $v\in\Lambda_{\omega_{\alpha},\infty}(\mathbb{D})$.
\end{CorA}

The following results are Lipschitz characterizations of  boundary functions and their harmonic extensions.

 \begin{ThmB}\label{Dya-2}{\rm(see \cite[Corollary 1 (ii)]{Dy1} \mbox{or} \cite[Theorem B]{P})}
Let $\omega$ be a regular majorant,  $f\in\mathscr{A}(\mathbb{D})$, and let the boundary function of $|f|$ belong to $\Lambda_{\omega,\infty}(\mathbb{T})$.
Then $f$ is in $\Lambda_{\omega,\infty}(\mathbb{D})$ if and only if $$P[|f|](z)-|f(z)|\leq M\omega(d_{\mathbb{D}}(z))$$ for some positive constant $M$.
\end{ThmB}

\begin{ThmC}\label{Dya-3}{\rm(\cite[Theorem 1]{Dy1})}
If $f\in\mathscr{A}(\mathbb{D})$ and if both $\omega$ and $\omega^{2}$ are regular majorants, then there is a positive constant $M$
such that \be\label{jjp-1}\frac{1}{M}\|f\|_{\Lambda_{\omega}(\overline{\mathbb{D}}),s}
\leq\sup_{z\in\mathbb{D}}\left\{\frac{\left(P[|f|^{2}](z)-|f(z)|^{2}\right)^{\frac{1}{2}}}{\omega\big(d_{\mathbb{D}}(z)\big)}\right\}\leq M\|f\|_{\Lambda_{\omega}(\overline{\mathbb{D}}),s}.\ee
\end{ThmC}

We remark that the norm appearing in the middle of the inequality (\ref{jjp-1}) can be regarded as an analogue of the so-called Garsia norm on the space $BMO$ (see \cite{Dy1} or \cite[Chapter VI]{Ga}).



 Recently, the equivalent norms  and Hardy-Littlewood type theorems on Lipschitz type spaces have attracted much attention of many authors
 (see \cite{AKM,AP-2017,MVM,CH-Riesz,CSR,CPR-2014,CPW-11,CP-2013,Dy2,Dy3,Dy4,P,Pav-2007,Q-W}). The first purpose of this paper is to use some  new techniques,
 in conjunction with some methods of  Dyakonov \cite{Dy1} and Pavlovi\'c \cite{P},
 to study the equivalent norms  and Hardy-Littlewood type theorems for $f\in\mathscr{H}(\mathbb{D})\cap\Lambda_{\omega,p}(\mathbb{D})$ and
 $f\in\mathscr{A}(\mathbb{D})\cap\Lambda_{\omega,p}(\mathbb{D})$, where $p\in[1,\infty]$. In the following, we will show that  Theorem A is equivalent to Corollary A. It is
 read as follows.

 \begin{Thm}\label{thm-1.0}
Suppose that $p\in[1,\infty]$ is a constant, $\omega$ is a fast majorant and $f=u+iv=h+\overline{g}\in\mathscr{H}(\mathbb{D})$  is continuous up to the boundary $\mathbb{T}$ of $\mathbb{D}$, where
 $h,~g\in\mathscr{A}(\mathbb{D})$ and $u, v$ are
real-valued  functions  in $\mathbb{D}$.
Then the following statements are equivalent.
\begin{enumerate}
\item[{\rm ($\mathscr{A}_{1}$)}] There is a positive constant $M$ such that, for $z\in\mathbb{D}$,
$$
\left\{
\begin{array}{ll}
\left(\int_{0}^{2\pi}\left(\mathscr{M}_{f}(ze^{i\eta})\right)^{p}d\eta\right)^{\frac{1}{p}}\leq M\frac{\omega\big(d_{\mathbb{D}}(z)\big)}{d_{\mathbb{D}}(z)}, & p\in [1,\infty), \\
\mathscr{M}_{f}(z) \leq M\frac{\omega\big(d_{\mathbb{D}}(z)\big)}{d_{\mathbb{D}}(z)}, & p=\infty ;
\end{array}
\right.
$$
\item[{\rm ($\mathscr{A}_{2}$)}] $f\in\Lambda_{\omega,p}(\overline{\mathbb{D}});$
\item[{\rm ($\mathscr{A}_{3}$)}] $h,~g\in\Lambda_{\omega,p}(\overline{\mathbb{D}});$
\item[{\rm ($\mathscr{A}_{4}$)}] $|h|,~|g|\in\Lambda_{\omega,p}(\overline{\mathbb{D}});$
\end{enumerate}
Furthermore, assume that $f$ is also
$K$-quasiregular in $\mathbb{D}$, where $K\in[1,\infty)$ is a constant.  Then, the conditions {\rm ($\mathscr{A}_{1}$)}$\sim${\rm ($\mathscr{A}_{4}$)} are equivalent to
the conditions {\rm ($\mathscr{A}_{5}$)}$\sim${\rm ($\mathscr{A}_{11}$)}.
\begin{enumerate}
\item[{\rm ($\mathscr{A}_{5}$)}] $|f|\in\Lambda_{\omega,p}(\overline{\mathbb{D}});$
\item[{\rm ($\mathscr{A}_{6}$)}] $u\in\Lambda_{\omega,p}(\overline{\mathbb{D}});$
\item[{\rm ($\mathscr{A}_{7}$)}] $|u|\in\Lambda_{\omega,p}(\overline{\mathbb{D}});$
\item[{\rm ($\mathscr{A}_{8}$)}] $v\in\Lambda_{\omega,p}(\overline{\mathbb{D}});$
\item[{\rm ($\mathscr{A}_{9}$)}] $|v|\in\Lambda_{\omega,p}(\overline{\mathbb{D}});$
\item[{\rm ($\mathscr{A}_{10}$)}] There is a positive constant $M$ such that, for $z\in\mathbb{D}$,
$$\mathscr{Y}_{u,q,p}(z)
\leq M\frac{\omega\big(d_{\mathbb{D}}(z)\big)}{d_{\mathbb{D}}(z)},$$ where  $p\vee q:=\{p,q:~0<q\leq p<\infty\}$
or  $p\wedge q:=\{p,q:~0<q<p=\infty\}$,
$$\mathscr{Y}_{u,q,p}(z)=
\left\{
\begin{array}{ll}
\left(\int_{0}^{2\pi}\left(\frac{\int_{\mathbb{D}\left(z,d_{\mathbb{D}}(z)/2\right)}\left(\mathscr{M}_{u}(\xi e^{i\eta})\right)^{q}dA(\xi)}
{\left|\mathbb{D}\left(z,d_{\mathbb{D}}(z)/2\right)\right|}
\right)^{\frac{p}{q}}d\eta\right)^{\frac{1}{p}}, & p\vee q, \\
\left(\frac{1}
{\left|\mathbb{D}\left(z,d_{\mathbb{D}}(z)/2\right)\right|}\int_{\mathbb{D}\left(z,d_{\mathbb{D}}(z)/2\right)}\left(\mathscr{M}_{u}(\xi)\right)^{q}dA(\xi)
\right)^{\frac{1}{q}}, & p\wedge q,
\end{array}
\right.
$$
 $\left|\mathbb{D}\big(z,d_{\mathbb{D}}(z)/2\big)\right|$
denotes the area of $\mathbb{D}\big(z,d_{\mathbb{D}}(z)/2\big)$,  and  $dA$ denotes  the Lebesgue area measure on $\mathbb{D}$;
\item[{\rm ($\mathscr{A}_{11}$)}] There is a positive constant $M$ such that, for $z\in\mathbb{D}$,
$$\mathscr{Y}_{v,q,p}(z)
\leq M\frac{\omega\big(d_{\mathbb{D}}(z)\big)}{d_{\mathbb{D}}(z)},$$ where $0<q\leq p<\infty$
or $0<q<p=\infty$.
\end{enumerate}
\end{Thm}

We remark that ($\mathscr{A}_{1}$)  is  equivalent to $f\in\Lambda_{\omega,p}(\mathbb{D})$ without the assumption ``$f$ is continuous on $\mathbb{T}$" in Theorem
\ref{thm-1.0}.
 The following result easily follows from Theorem \ref{thm-1.0}.

\begin{Cor}\label{An-th-0.1}
Suppose that $p\in[1,\infty]$ is a constant, $\omega$ is a fast majorant and $f=u+iv\in\mathscr{A}(\mathbb{D})$, where $u$ and $v$ are
real-valued functions in $\mathbb{D}$ that are continuous up to $\mathbb{T}$.
Then the following statements are equivalent.
\begin{enumerate}
\item[{\rm ($\mathscr{B}_{1}$)}] There is a positive constant $M$ such that, for $z\in\mathbb{D}$,
$$
\left\{
\begin{array}{ll}
\left(\int_{0}^{2\pi}|f'(ze^{i\eta})|^{p}d\eta\right)^{\frac{1}{p}}\leq M\frac{\omega\big(d_{\mathbb{D}}(z)\big)}{d_{\mathbb{D}}(z)}, & p\in [1,\infty)\\
|f'(z)|\leq M\frac{\omega\big(d_{\mathbb{D}}(z)\big)}{d_{\mathbb{D}}(z)}, & p=\infty;
\end{array}
\right.
$$
\item[{\rm ($\mathscr{B}_{2}$)}] $f\in\Lambda_{\omega,p}(\overline{\mathbb{D}});$
\item[{\rm ($\mathscr{B}_{3}$)}] $|f|\in\Lambda_{\omega,p}(\overline{\mathbb{D}});$
\item[{\rm ($\mathscr{B}_{4}$)}] $u\in\Lambda_{\omega,p}(\overline{\mathbb{D}});$
\item[{\rm ($\mathscr{B}_{5}$)}] $|u|\in\Lambda_{\omega,p}(\overline{\mathbb{D}});$
\item[{\rm ($\mathscr{B}_{6}$)}] $v\in\Lambda_{\omega,p}(\overline{\mathbb{D}});$
\item[{\rm ($\mathscr{B}_{7}$)}] $|v|\in\Lambda_{\omega,p}(\overline{\mathbb{D}});$
\item[{\rm ($\mathscr{B}_{8}$)}] There is a positive constant $M$ such that, for $z\in\mathbb{D}$,
$$\mathscr{Y}_{u,q,p}(z)
\leq M\frac{\omega\big(d_{\mathbb{D}}(z)\big)}{d_{\mathbb{D}}(z)},$$ where $0<q\leq p<\infty$
or $0<q<p=\infty$;
\item[{\rm ($\mathscr{B}_{9}$)}] There is a positive constant $M$ such that, for $z\in\mathbb{D}$,
$$\mathscr{Y}_{v,q,p}(z)
\leq M\frac{\omega\big(d_{\mathbb{D}}(z)\big)}{d_{\mathbb{D}}(z)},$$ where $0<q\leq p<\infty$
or $0<q<p=\infty$.
\end{enumerate}
\end{Cor}

We remark that Corollary \ref{An-th-0.1} is an   improvement and extension of Theorem A and Corollary A.

In order to extend Theorems B and C, we need to
establish another version of Hardy-Littlewood type theorem on complex-valued  harmonic functions. Before presenting our results, let us recall the classical  Hardy-Littlewood theorem
of complex-valued harmonic functions  as follows (cf.
 \cite{H-L-31,H-L,Pav-2008,Pri}).

 \begin{ThmD}\label{H-L}{\rm (Hardy-Littlewood's Theorem)}
 If $\varphi\in\Lambda_{\omega_{\beta},\infty}(\mathbb{T})$, then
$P[\varphi]\in\Lambda_{\omega_{\beta},\infty}(\overline{\mathbb{D}})$, where $\beta\in(0,1)$ and $\omega_{\beta}(t)=t^{\beta}$ for $t\geq0$.
 \end{ThmD}

 Nolder and Oberlin generalized Theorem D, and established a
 Hardy-Littlewood theorem for a  differentiable majorant (see \cite[Theorem 1.5]{No}).
 Later, Dyakonov \cite{Dy1} generalized Theorem D to complex-valued harmonic functions as follows.
 For  some related studies, we refer the readers to \cite{AKM,Ch-18,CSR,Dy2,Dy3,GK,No,P,Pav-2007,P-08} for details.


 \begin{ThmE}\label{Th-B}{\rm (\cite[Lemma 4]{Dy1})}  Let $\omega$ be a regular majorant.
  If   $\varphi\in\Lambda_{\omega,\infty}(\mathbb{T})$,
then $P[\varphi]\in\Lambda_{\omega,\infty}(\overline{\mathbb{D}})$.
\end{ThmE}

Replacing a regular majorant by a weaker one,  we use some new techniques to improve and extend
 Theorem E  into the following form.

 \begin{Thm}\label{thm-5.0}
Let $\omega$ be a majorant, and let $p\in[1,\infty]$ be a constant. 



\begin{enumerate}
\item[{\rm ($\mathscr{C}_{1}$)}] If $\varphi\in\Lambda_{\omega,p}(\mathbb{T})$ and
$\varphi$ is continuous on $\mathbb{T}$, then $P[\varphi]\in\Lambda_{\omega,p}(\overline{\mathbb{D}})$;
\item[{\rm ($\mathscr{C}_{2}$)}] There is a positive constant $M$ such that for all $\delta\in[0,\pi]$,
$$\delta\int_{\delta}^{\pi}\frac{\omega(t)}{t^{2}}dt\leq\,M\omega(\delta).$$
\end{enumerate}
Then $(\mathscr{C}_{2})\Rightarrow(\mathscr{C}_{1})$ for $p\in[1,\infty]$, and $(\mathscr{C}_{1})\Leftrightarrow(\mathscr{C}_{2})$ for $p=\infty$.
\end{Thm}

 In particular, if $\omega$ is a slow majorant,
then the following result easily follows from Theorem \ref{thm-5.0}.

\begin{Cor}\label{cor-0.11}
Let $\omega$ be a slow majorant, and let $p\in[1,\infty]$ be a constant.
If $\varphi\in\Lambda_{\omega,p}(\mathbb{T})$
and $\varphi$ is continuous on $\mathbb{T}$, then $P[\varphi]\in\Lambda_{\omega,p}(\overline{\mathbb{D}})$.
\end{Cor}

By using Theorem \ref{thm-5.0}, we improve and generalize  Theorem  B into the following form.

 \begin{Thm}\label{t-3}
Suppose that $p\in[1,\infty]$ is a constant, $\omega$ is a regular majorant and $f\in\mathscr{H}(\mathbb{D})$ is continuous up to the boundary $\mathbb{T}$ of $\mathbb{D}$.
If the boundary function of $|f|\in\Lambda_{\omega,p}(\mathbb{T})$ and $f$ is
$K$-quasiregular in $\mathbb{D}$, where $K\in[1,\infty)$ is a constant,
then the following statements are equivalent.

\begin{enumerate}
\item[{\rm ($\mathscr{D}_{1}$)}] $f\in\Lambda_{\omega,p}(\mathbb{D})$;
\item[{\rm ($\mathscr{D}_{2}$)}] There is a positive constant $M$ such that
 $$
 \left\{
 \begin{array}{ll}
 \left(\int_{0}^{2\pi}\left(P[|f|](ze^{i\eta})-|f(ze^{i\eta})|\right)^{p}d\eta\right)^{\frac{1}{p}}
\leq M\omega(d_{\mathbb{D}}(z)), &p\in[1,\infty),
\\
P[|f|](z)-|f(z)|\leq M\omega(d_{\mathbb{D}}(z)), &p=\infty.
\end{array}
\right.
$$
\end{enumerate}
\end{Thm}

 \begin{Rem}
We remark that $|f|^{\lambda}$ is subharmonic, and therefore the Poisson integral, $P[|f|^{\lambda}]$,
of the boundary function of $|f|^{\lambda}$, is equal to the smallest harmonic majorant of $|f|^{\lambda}$
in $\mathbb{D}$, where $K\in[1,\infty)$, $\lambda\in\left[1-1/K^{2}, \infty\right)$ and $f\in\mathscr{H}(\mathbb{D})$ is a
$K$-quasiregular mapping $($see \cite{KP-08}$)$. In particular, $P[|f|^{\lambda}]-|f|^{\lambda}\geq0$ in $\mathbb{D}$, which implies that
$P[|f|]-|f|\geq0$ in $\mathbb{D}$ holds in Theorem \ref{t-3}.
\end{Rem}

From Theorem \ref{t-3}, we obtain the following result.

\begin{Cor}\label{yy-cor}
Suppose that $p\in[1,\infty]$ is a constant, $\omega$ is a regular majorant and $f\in\mathscr{A}(\mathbb{D})$.
If the boundary function of $|f|\in\Lambda_{\omega,p}(\mathbb{T})$, then the following statements are equivalent.

\begin{enumerate}
\item[{\rm ($\mathscr{E}_{1}$)}] $f\in\Lambda_{\omega,p}(\mathbb{D})$;
\item[{\rm ($\mathscr{E}_{2}$)}] There is a positive constant $M$ such that
$$
 \left\{
 \begin{array}{ll}
 \left(\int_{0}^{2\pi}\left(P[|f|](ze^{i\eta})-|f(ze^{i\eta})|\right)^{p}d\eta\right)^{\frac{1}{p}}
\leq M\omega(d_{\mathbb{D}}(z)), &p\in[1,\infty),
\\
P[|f|](z)-|f(z)|\leq M\omega(d_{\mathbb{D}}(z)), &p=\infty.
\end{array}
\right.
$$
\end{enumerate}
\end{Cor}

 Furthermore, if we replace $|f|$ by $|f|^{2}$  in  ($\mathscr{E}_{2}$) of Corollary \ref{yy-cor},  then we obtain the following result which is an extension of Theorem C.

\begin{Thm}\label{thm-8}
If $f\in\mathscr{A}(\mathbb{D})$ and if both $\omega$ and $\omega^{2}$ are regular majorants, then
$f\in\Lambda_{\omega,p}(\overline{\mathbb{D}})$ if and only if
there is a positive constant $M$
such that

\beqq\label{df-1-1}
 \left\{
 \begin{array}{ll}
 \left(\int_{0}^{2\pi}\left(P[|f|^{2}](ze^{i\eta})-|f(ze^{i\eta})|^{2}\right)^{\frac{p}{2}}d\eta\right)^{\frac{1}{p}}
\leq M\omega(d_{\mathbb{D}}(z)), &p\in[2,\infty),
\\
\left(P[|f|^{2}](z)-|f(z)|^{2}\right)^{\frac{1}{2}}\leq M\omega(d_{\mathbb{D}}(z)), &p=\infty.
\end{array}
\right.
\eeqq
 \end{Thm}

 We remark that if we take $p=\infty$, then  Theorem \ref{thm-8}  coincides with Theorem C.

 \subsection{Applications of equivalent norms and Hardy-Littlewood type theorems}
Let $\phi$ be a analytic function of $\mathbb{D}$
into itself. We define
the composition operator $C_{\phi}$ on $\mathscr{H}(\mathbb{D})$ by $C_{\phi}(f)(z)=f(\phi(z))$ for $z\in\mathbb{D}$. For two sets  $X$,
$Y\subset \mathscr{H}(\mathbb{D})$, the set of all $\phi$ for which $C_{\phi}(X)\subset Y$ is
denoted by $\mathscr{F}(X,Y)$. In \cite{Sha}, Shapiro gave some compete characterizations of compact
composition operators of some analytic function spaces. Recently, the studies of  composition operators on holomorphic and harmonic functions
have been attracted much attention of many mathematicians (see \cite{C-H,CHZ2022MZ,CPR,P-08,Pav-2008,P-R-0,P-R-1,Sha,Z2}).

 The second purpose of this paper is to apply Theorems \ref{thm-1.0} and \ref{thm-5.0} to study
the composition operators between the Lipschitz type spaces. In the following, for $x\in\mathbb{R}$, let $$\{x\}_{+}=\max\{x,0\}.$$
By using Theorems \ref{thm-1.0} and \ref{thm-5.0}, we obtain the following result.

 \begin{Thm}\label{th-3}
Suppose that $p\in[1,\infty)$ is a constant, and $\omega_{1}$ and $\omega_{2}$ are fast majorants.
Let $\phi$ be an analytic function of $\mathbb{D}$ into itself.
Then the following conditions are equivalent.
\begin{enumerate}
\item[{\rm $(\mathscr{F}_{1})$}] $\phi\in\mathscr{F}\big(\Lambda_{\omega_{1},\infty}(\mathbb{D})\cap\mathscr{H}(\mathbb{D}),\Lambda_{\omega_{2},p}(\mathbb{D})\cap\mathscr{H}(\mathbb{D})\big)$;

\item[{\rm $(\mathscr{F}_{2})$}] There is a positive constant $M$ such that, for $z\in\mathbb{D}$,
$$\left(\int_{0}^{2\pi}\left(|\phi'(ze^{i\theta})|
\frac{\omega_{1}\big(d_{\mathbb{D}}(\phi(ze^{i\theta}))\big)}{d_{\mathbb{D}}(\phi(ze^{i\theta}))}\right)^{p}d\theta\right)^{\frac{1}{p}}
\leq M\frac{\omega_{2}\big(d_{\mathbb{D}}(z)\big)}{d_{\mathbb{D}}(z)};$$
\end{enumerate}

Furthermore, assume that
$\phi$ is continuous in $\overline{\mathbb{D}}$,
$\omega_{1}\in\mathscr{S}$ and $\omega_{2}$ is regular, where $\mathscr{S}$ denotes the set consisting of those majorant $\omega_{1}$
for which $\omega_{1}$ is differentiable on $(0,1]$ and $\omega_{1}'$
is also non-increasing on $(0,1]$ with $\sup_{t\in (0,1]}\frac{\omega_1(t)}{t\omega_1'(t)}<\infty$.

Then, the conditions {\rm $(\mathscr{F}_{1})$} and {\rm $(\mathscr{F}_{2})$} are equivalent to the conditions
 $(\mathscr{F}_{3})$$\sim$$(\mathscr{F}_{5})$.
\begin{enumerate}
\item[{\rm $(\mathscr{F}_{3})$}]  $\omega_{1}(d_{\mathbb{D}}(\phi))\in\Lambda_{\omega_{2},p}(\overline{\mathbb{D}});$

\item[{\rm $(\mathscr{F}_{4})$}] $\omega_{1}(d_{\mathbb{D}}(\phi))\in\Lambda_{\omega_{2},p}(\mathbb{T})$ and there is a positive constant $M$
such that, for $r\in(0,1)$,
$$\left(\int_{0}^{2\pi}\left\{\omega_{1}(d_{\mathbb{D}}(\phi(re^{i\theta})))-\omega_{1}(d_{\mathbb{D}}(\phi(e^{i\theta})))\right\}_{+}^{p}d\theta\right)^{\frac{1}{p}}
\leq\,M\omega_{2}(d_{\mathbb{D}}(r));$$
\item[{\rm $(\mathscr{F}_{5})$}]
$\omega_{1}(d_{\mathbb{D}}(\phi))\in\Lambda_{\omega_{2},p}(\mathbb{T})$ and there is a positive constant $M$
such that, for $r\in(0,1)$,
$$\left(\int_{0}^{2\pi}\left(\omega_{1}(d_{\mathbb{D}}(\phi(re^{i\theta})))-P[\omega_{1}(d_{\mathbb{D}}(\phi))](re^{i\theta})\right)^{p}d\theta\right)^{\frac{1}{p}}
\leq\,M\omega_{2}(d_{\mathbb{D}}(r)).$$
\end{enumerate}
\end{Thm}

 \begin{Rem}
We remark that $$\omega_{1}(d_{\mathbb{D}}(\phi(re^{i\theta})))-P[\omega_{1}(d_{\mathbb{D}}(\phi))](re^{i\theta})\geq0$$ in Theorem \ref{th-3} because
$\omega_{1}(d_{\mathbb{D}}(\phi))$ is superharmonic in $\mathbb{D}$ {\rm (see Lemma \ref{le-9})}. In particular, Theorem \ref{th-3} is also true for $p=\infty$.
 The proof method used to prove Theorem \ref{th-3} for $p\in[1,\infty)$ is still valid for $p=\infty$.
\end{Rem}

For $t\geq0$, let $\omega_{\alpha}(t)=t^{\alpha}$ and $\omega_{\beta}(t)=t^{\beta}$, where $\alpha, \beta\in(0,1]$ are constants.
In particular, if we take $\omega_{1}=\omega_{\alpha}$ and $\omega_{2}=\omega_{\beta}$ in Theorem \ref{th-3}, then
we obtain the following result. 

\begin{Cor}
Suppose that $p\in[1,\infty)$ is a constant and
 $\phi$ is an analytic function of $\mathbb{D}$ into itself.
Then the following conditions are equivalent.
\begin{enumerate}
\item[{\rm $(\mathscr{G}_{1})$}] $\phi\in\mathscr{F}\big(\Lambda_{\omega_{\alpha},\infty}(\mathbb{D})\cap\mathscr{H}(\mathbb{D}),\Lambda_{\omega_{\beta},p}(\mathbb{D})\cap\mathscr{H}(\mathbb{D})\big)$;

\item[{\rm $(\mathscr{G}_{2})$}] There is a positive constant $M$ such that, for $z\in\mathbb{D}$,
$$\left(\int_{0}^{2\pi}\left(|\phi'(ze^{i\theta})|
{\big(d_{\mathbb{D}}(\phi(ze^{i\theta}))\big)^{\alpha-1}}\right)^{p}d\theta\right)^{\frac{1}{p}}
\leq M\big(d_{\mathbb{D}}(z)\big)^{\beta-1};$$
\end{enumerate}

Furthermore, assume that
$\phi$ is continuous in $\overline{\mathbb{D}}$ and $\beta\neq1$.
Then, the conditions {\rm $(\mathscr{G}_{1})$} and {\rm $(\mathscr{G}_{2})$} are equivalent to the conditions
 $(\mathscr{G}_{3})$$\sim$$(\mathscr{G}_{5})$.
\begin{enumerate}
\item[{\rm $(\mathscr{G}_{3})$}]  ${\big(d_{\mathbb{D}}(\phi)\big)^{\alpha}}\in\Lambda_{\omega_{\beta},p}(\overline{\mathbb{D}});$

\item[{\rm $(\mathscr{G}_{4})$}] ${\big(d_{\mathbb{D}}(\phi)\big)^{\alpha}}\in\Lambda_{\omega_{\beta},p}(\mathbb{T})$ and there is a positive constant $M$
such that, for $r\in(0,1)$,
$$\left(\int_{0}^{2\pi}\left\{(d_{\mathbb{D}}(\phi(re^{i\theta})))^{\alpha}-
(d_{\mathbb{D}}(\phi(e^{i\theta})))^{\alpha}\right\}_{+}^{p}d\theta\right)^{\frac{1}{p}}
\leq\,M(d_{\mathbb{D}}(r))^{\beta};$$
\item[{\rm $(\mathscr{G}_{5})$}]
$(d_{\mathbb{D}}(\phi))^{\alpha}\in\Lambda_{\omega_{\beta},p}(\mathbb{T})$ and there is a positive constant $M$
such that, for $r\in(0,1)$,
$$\left(\int_{0}^{2\pi}\left((d_{\mathbb{D}}(\phi(re^{i\theta})))^{\alpha}-P[(d_{\mathbb{D}}(\phi))^{\alpha}](re^{i\theta})\right)^{p}d\theta\right)^{\frac{1}{p}}
\leq\,M(d_{\mathbb{D}}(r))^{\beta}.$$
\end{enumerate}
\end{Cor}

The proofs of Theorems \ref{Open-1},  \ref{thm-1.0}, \ref{thm-5.0}, \ref{t-3} and \ref{thm-8} will be presented in Sec. \ref{sec2}, and
the proof of Theorem \ref{th-3}  will be given in Sec. \ref{sec3}.

 \section{Equivalent norms  and Hardy-Littlewood type theorems on Lipschitz type spaces}\label{sec2}

\subsection{The proof of Theorem \ref{Open-1}}
Let $f$ be analytic in $\mathbb{D}$. We will show that
if $|f|\in\Lambda_{\omega_{\alpha},\infty}(\mathbb{D})$ for $\alpha>1$, then $f$ is a constant function, and $f\in\Lambda_{\omega_{\alpha},\infty}(\mathbb{D}).$
It follows from  the assumptions that, for all $z,w\in\mathbb{D}$, there is a positive constant $M$ such that
\be\label{o-p-1}\big||f(z)|-|f(w)|\big|\leq M|z-w|^{\alpha}.\ee
Let $E=\{\varsigma\in\mathbb{D}:~f(\varsigma)=0\}$.
We may assume that $\mathbb{D}\setminus E \neq \emptyset$.
For $z\in\mathbb{D}\setminus E$,
let $w=z+re^{i\theta}\in\mathbb{D}$, where $r\in(0,d_{\mathbb{D}}(z))$   and $\theta\in[0,2\pi]$.
Then, by (\ref{o-p-1}), we have
\beqq
0\leq\lim_{r\rightarrow0^{+}}\frac{\big||f(z)|-|f(z+re^{i\theta})|\big|}{r}\leq M\lim_{r\rightarrow0^{+}}r^{\alpha-1}=0,
\eeqq
which implies that
\beqq
0&=&\max_{\theta\in[0,2\pi]}\lim_{r\rightarrow0^{+}}\frac{\big||f(z)|-|f(z+re^{i\theta})|\big|}{r}\\
&=&\max_{\theta\in[0,2\pi]}\big||f(z)|_{x}\cos\theta+|f(z)|_{y}\sin\theta\big|\\
&=&\max_{\theta\in[0,2\pi]}\frac{1}{2}\left|\big(|f(z)|_{x}+i|f(z)|_{y}\big)e^{-i\theta}+\big(|f(z)|_{x}-i|f(z)|_{y}\big)e^{i\theta}\right|\\
&=&\big||f(z)|_{z}\big|+\big||f(z)|_{\overline{z}}\big|\\
&=&{|f'(z)|}.
\eeqq Hence $f'(z)\equiv0$ for all $z\in\mathbb{D}\setminus E$.
Since $\mathbb{D}\setminus E$ is a nonempty open subset of $\mathbb{D}$, $f'(z)\equiv0$ for all $z\in\mathbb{D}$.
Therefore, $f$ is a constant function in $\mathbb{D}$, and  $f\in\Lambda_{\omega_{\alpha},\infty}(\mathbb{D}).$
The proof of this theorem is complete.
\qed

%

\begin{LemF}{\rm (\cite[p.19]{Zy})}\label{Mink}
For $\nu\geq1$, Minkowski's inequality in infinite form is
$$\left(\int_{A_{1}}\left|\int_{B_{1}}\mathscr{X}(\zeta,\xi)d\mu_{\xi}\right|^{\nu}d\mu_{\zeta}\right)^{\frac{1}{\nu}}\leq
\int_{B_{1}}\left(\int_{A_{1}}\left|\mathscr{X}(\zeta,\xi)\right|^{\nu}d\mu_{\zeta}\right)^{\frac{1}{\nu}}d\mu_{\xi},$$
where $A_{1}$ and $B_{1}$ are measurable sets with positive measures $d\mu_{\zeta}$  and $d\mu_{\xi}$, respectively, and $\mathscr{X}$
is integrable on $A_{1}\times B_{1}$.
\end{LemF}

It follows from Lewy's Theorem  that $f\in\mathscr{H}(\mathbb{D})$ is locally univalent and sense-preserving in
$\mathbb{D}$ if and only if $J_{f}>0$ in $\mathbb{D}$, which means that
$f_{z}\neq 0$ in $\mathbb{D}$ and the second complex dilatation
$\mu_{f} =\overline{f_{\bar{z}}}/f_{z}$
has the property that $|\mu_{f} (z)|<1$ in $\mathbb{D}$ (see \cite{Lewy}).
From the definition of $K$-quasiregular mappings, we know that if $f\in\mathscr{H}(\mathbb{D})$ is a harmonic $K$-quasiregular mapping if and only if $J_{f}(z)>0$
and $\mathscr{M}_{f}^{2}(z)\leq KJ_{f}(z)$  for all $z\in\mathbb{D}$.
By using almost the same proof method of \cite[Proposition 3.1]{KM}, we get the following result.

\begin{Lem}\label{mate-07}
For $K\geq1$, let $f$ be a harmonic $K$-quasiregular mapping of $\mathbb{D}$ into itself. Then, for $z\in\mathbb{D}$,
$$\mathscr{M}_{f}(z)\leq K\frac{1-|f(z)|^{2}}{1-|z|^{2}}.$$
\end{Lem}

 Using Lemma \ref{mate-07},
we obtain the following lemma.

\begin{Lem}\label{mate-07b}
For $K\geq1$, let $f$ be a harmonic $K$-quasiregular mapping of $\mathbb{D}$ into $\mathbb{C}$. Then, for $z\in\mathbb{D}$ and $0<\delta<d_{\mathbb{D}}(z)$,
\begin{equation}
\label{mate07c}
\mathscr{M}_{f}(z)\leq \frac{2K\left(M_{z, \delta}-|f(z)|\right)}{\delta},
\end{equation}
where $M_{z,\delta}=\sup\{|f(w)|:~|w-z|<\delta \}$.
\end{Lem}

\begin{proof}
We may assume that $f$ is not a constant.
For a fixed point $z\in\mathbb{D}$, let

$$F(\zeta)=\frac{f(z+\delta \zeta )}{M_{z,\delta}},~\zeta\in\mathbb{D}.$$
 Elementary calculations lead to
\be\label{eq-rtp-1}
F(0)=\frac{f(z)}{M_{z,\delta}}~\mbox{and}~\mathscr{M}_{F}(0)=
\frac{\delta}{M_{z, \delta}}\mathscr{M}_{f}(z).
\ee
Since $F$ is a harmonic $K$-quasiregular mapping
of $\mathbb{D}$ into itself, by (\ref{eq-rtp-1}) and Lemma \ref{mate-07}, we see that
\beqq
\frac{\delta}{M_{z,\delta}}\mathscr{M}_{f}(z)&=&\mathscr{M}_{F}(0)\leq K\left(1-|F(0)|^{2}\right)\leq2K(1-|F(0)|)\\
&=&2K\left(1-\frac{|f(z)|}{M_{z,\delta}}\right).
\eeqq
Consequently,
\[
\delta\mathscr{M}_{f}(z)\leq 2K\left(M_{z,\delta}-|f(z)|\right),
\]
which implies (\ref{mate07c}).
\end{proof}

 \begin{Lem}\label{mate-07d}
Suppose that $p\in[1,\infty)$ and $0<\varepsilon<1-|z|$,
$z\in \mathbb{D}$.
For $K\geq1$, let $f$ be a harmonic $K$-quasiregular mapping of $\mathbb{D}$ into $\mathbb{C}$. Then, for $z\in\mathbb{D}$,
\[
\mathscr{M}_{f}^{p}(z)\leq \frac{M}{\varepsilon^{p+2}}
\int_{\mathbb{D}(z,\varepsilon)}\{|f(w)|-|f(z)|\}_+^p dm(w),
\]
where $M$ is a positive constant which depends only on $p$ and $K$,
and $dm=dA/\pi$ $($see Theorem \ref{thm-1.0}$)$.
\end{Lem}

\begin{proof}
Let $z\in \mathbb{D}$ be fixed with $0<\varepsilon<1-|z|$.
The function $\{ |f(w)|-|f(z)|\}_+$, $w\in \mathbb{D}$,
is subharmonic in $\mathbb{D}$.
By applying Lemma \ref{mate-07b} for $\delta=\varepsilon/2$,
we have
\begin{equation}
\label{eq-3-1}
\mathscr{M}_{f}(z)\leq \frac{4K\left(M_{z,\varepsilon/2}-|f(z)|\right)}{\varepsilon}
=\frac{4K}{\varepsilon}\sup_{w\in \mathbb{D}(z,\varepsilon/2)}\{ |f(w)|-|f(z)|\}_+
\end{equation}
Since $\{ |f(w)|-|f(z)|\}_{+}^p$ is subharmonic on $\mathbb{D}$,
for $w\in \mathbb{D}(z,\delta)$,
we have
\[
\{ |f(w)|-|f(z)|\}_{+}^p
\leq
\frac{1}{\delta^2}\int_{\mathbb{D}(z,\varepsilon)}\{ |f(\eta)|-|f(z)|\}_{+}^{p}dm(\eta),
\]
which combined with (\ref{eq-3-1}) implies that
\[
\mathscr{M}_{f}^p(z)\leq
\frac{4^{p+1}K^p}{\varepsilon^{p+2}}\int_{\mathbb{D}(z,\varepsilon)}\{ |f(\eta)|-|f(z)|\}_{+}^{p}dm(\eta).
\]
This completes the proof.
\end{proof}

From \cite[Theorem 1.8]{KM-2011},
we obtain the following result.

\begin{Lem}\label{real-harmonic-07}
Let $f$ be a real harmonic mapping of $\mathbb{D}$ into $(-1,1)$. Then,
there exists a constant $M>0$ which is independent of $f$
such that for $z\in\mathbb{D}$,
$$\mathscr{M}_{f}(z)\leq M\frac{1-|f(z)|^{2}}{1-|z|^{2}}.$$
\end{Lem}

 Using Lemma \ref{real-harmonic-07}
 and an argument similar to that in the proof of Lemma \ref{mate-07b},
we obtain the following lemma

\begin{Lem}\label{real-harmonic-07b}
Let $f$ be a real harmonic mapping of $\mathbb{D}$ into $\mathbb{R}$. Then, there exists a constant $M>0$ which is independent of $f$
such that
for $z\in\mathbb{D}$ and $0<\delta<d_{\mathbb{D}}(z)$,
\[
\mathscr{M}_{f}(z)\leq \frac{2M\left(M_{z, \delta}-|f(z)|\right)}{\delta},
\]
where $M_{z,\delta}=\sup\{|f(w)|:~|w-z|<\delta \}$.
\end{Lem}

 \begin{Lem}\label{real-harmonic-07d}
Suppose that $p\in[1,\infty)$ and $0<\varepsilon<1-|z|$,
$z\in \mathbb{D}$.
Let $f$ be a real harmonic mapping of $\mathbb{D}$ into $\mathbb{R}$.
Then, for $z\in\mathbb{D}$,
\[
\mathscr{M}_{f}^{p}(z)\leq \frac{M}{\varepsilon^{p+2}}
\int_{\mathbb{D}(z,\varepsilon)}\{|f(w)|-|f(z)|\}_+^p dm(w),
\]
where $M$ is a positive constant which depends only on $p$,
and $dm=dA/\pi$ $($see Theorem \ref{thm-1.0}$)$.
\end{Lem}

\subsection{The proof of Theorem \ref{thm-1.0}}
\noindent $\mathbf{Case~1.}$We first give a proof in the case $p\in[1,\infty)$. We split the proof of this case into twelve steps.

\noindent $\mathbf{Step~1.}$ ``$(\mathscr{A}_{2})\Rightarrow(\mathscr{A}_{1})$".\\

 For any fixed $z\in\mathbb{D}$, let $r=2d_{\mathbb{D}}(z)/3$. By assumption, we see that
there is a positive constant $M$ such that
\be\label{eq-0.01k}
\left(\int_{0}^{2\pi}|f(e^{i\theta}z)-f(e^{i\theta}\xi)|^{p}d\theta\right)^{\frac{1}{p}}\leq M\omega(|z-\xi|)\leq M\omega(r)
\ee
for $\xi\in\overline{\mathbb{D}(z,r)}$.
For $\theta\in[0,2\pi]$ and $e^{i\theta}\xi\in\mathbb{D}(e^{i\theta}z,r)$, we have
$$f(e^{i\theta}\xi)=\frac{1}{2\pi}\int_{0}^{2\pi}P_{r}(\xi,e^{i\eta})f(ze^{i\theta}+e^{i\theta}re^{i\eta})d\eta,$$
where $$P_{r}(\xi,e^{i\eta})=\frac{r^{2}-|\xi-z|^{2}}{|re^{i\eta}-(\xi-z)|^{2}}.$$
Elementary calculations lead to
\beqq
\frac{\partial}{\partial\xi}P_{r}(\xi,e^{i\eta})&=&\frac{\overline{z}-\overline{\xi}}{|re^{i\eta}-(\xi-z)|^{2}}\\
&&+
\frac{(r^{2}-|\xi-z|^{2})(re^{-i\eta}-(\overline{\xi}-\overline{z}))}{|re^{i\eta}-(\xi-z)|^{4}}.
\eeqq
Then, for $e^{i\theta}\xi\in\mathbb{D}(e^{i\theta}z,3r/4)$, we have
\be\label{eq-0.1k}
\left|\frac{\partial}{\partial\xi}P_{r}(\xi,e^{i\eta})\right|\leq\frac{40}{r}.
\ee
From the similar proof process,  we obtain the following inequality
\be\label{eq-0.2k}
\left|\frac{\partial}{\partial\overline{\xi}}P_{r}(\xi,e^{i\eta})\right|\leq\frac{40}{r}.
\ee
It follows from (\ref{eq-0.1k}),  (\ref{eq-0.2k}) and H\"{o}lder's inequality that
\beq\label{eq-0.3k}
\big(\mathscr{M}_{f}(e^{i\theta}\xi)\big)^{p}&\leq&\int_{0}^{2\pi}\left(\mathscr{M}_{P_{r}}(\xi, e^{i\eta})\right)^p
\left|\mathscr{P}_{f}(\theta,\eta)\right|^{p}\frac{d\eta}{2\pi}\\ \nonumber
&\leq&\frac{80^{p}}{r^{p}}\int_{0}^{2\pi}
\left|\mathscr{P}_{f}(\theta,\eta)\right|^{p}\frac{d\eta}{2\pi},
\eeq where $\mathscr{P}_{f}(\theta,\eta)=f(ze^{i\theta}+e^{i\theta}re^{i\eta})-f(ze^{i\theta})$.
By taking $\xi=z$ in (\ref{eq-0.3k}) and integrating both sides of the inequality (\ref{eq-0.3k}) with respect to $\theta$ from $0$ to $2\pi$, we obtain  from (\ref{eq-0.01k}) that
\beqq
\int_{0}^{2\pi}\left(\mathscr{M}_{f}(ze^{i\theta})\right)^{p}d\theta
&\leq&
\frac{(80)^{p}}{r^{p}}\int_{0}^{2\pi}\left(\int_{0}^{2\pi}
\big|\mathscr{P}_{f}(\theta,\eta)\big|^{p}\frac{d\eta}{2\pi}\right) d\theta
\\
&=&
\frac{(80)^{p}}{r^{p}}\int_{0}^{2\pi}\left(\int_{0}^{2\pi}
\big|\mathscr{P}_{f}(\theta,\eta)\big|^{p}d\theta\right) \frac{d\eta}{2\pi}
\\
&\leq&
(80M)^p\left(\frac{\omega\left(\frac{2}{3}d_{\mathbb{D}}(z)\right)}{\frac{2}{3}d_{\mathbb{D}}(z)}\right)^{p}\\
 &\leq& (120M)^{p}\left(\frac{\omega\big(d_{\mathbb{D}}(z)\big)}{d_{\mathbb{D}}(z)}\right)^{p}.
 \eeqq

\noindent $\mathbf{Step~2.}$ ``$(\mathscr{A}_{1})\Rightarrow(\mathscr{A}_{2})$".\\

Since $\mathbb{D}$ is a
$\Lambda_{\omega}$-extension domain,
we see that, for each pair of points
$z_{1},z_{2}\in \mathbb{D}$, there is a  rectifiable curve $\gamma\subset \mathbb{D}$ joining $z_{1}$ to $z_{2}$
such that
\be\label{eq-0.4k}
\int_{\gamma}\frac{\omega\big(d_{\mathbb{D}}(z)\big)}{d_{\mathbb{D}}(z)}ds(z)\leq\,M\omega(|z_{1}-z_{2}|),
\ee
where $M$ is a positive constant.
By Lemma F, we have
$$\left(\int_{0}^{2\pi}\left(\int_{\gamma}\mathscr{M}_{f}(ze^{i\eta})ds(z)\right)^{p}d\eta\right)^{\frac{1}{p}}
\leq\int_{\gamma}\left(\int_{0}^{2\pi}\left(\mathscr{M}_{f}(ze^{i\eta})\right)^{p}d\eta\,\right)^{\frac{1}{p}}ds(z),$$
which, together with the assumption and (\ref{eq-0.4k}), implies that there is a positive constant $M$  such that
\beqq
\mathcal{L}_{p}[f](z_{1},z_{2})&\leq&
\left(\int_{0}^{2\pi}\left(\int_{\gamma}\mathscr{M}_{f}(ze^{i\eta})ds(z)\right)^{p}d\eta\right)^{\frac{1}{p}}\\
&\leq&
\int_{\gamma}\left(\int_{0}^{2\pi}\left(\mathscr{M}_{f}(ze^{i\eta})\right)^{p}d\eta\right)^{\frac{1}{p}}ds(z)\\
&\leq& M\int_{\gamma}\frac{\omega\big(d_{\mathbb{D}}(z)\big)}{d_{\mathbb{D}}(z)}ds(z)\\
&\leq&\,M^{2}\omega(|z_{1}-z_{2}|).
\eeqq
Since $f$ is continuous on $\overline{\mathbb{D}}$,
the above inequality implies that
$f\in\Lambda_{\omega,p}(\overline{\mathbb{D}})$.

\noindent $\mathbf{Step~3.}$ ``$(\mathscr{A}_{2})\Rightarrow(\mathscr{A}_{3})$".\\

Let $f=h+\overline{g}=u+iv$, where $h=u_{1}+iv_{1}$ and $g=u_{2}+iv_{2}$. Then $f\in\Lambda_{\omega,p}(\overline{\mathbb{D}})$ if and only if
$u,~v\in\Lambda_{\omega,p}(\overline{\mathbb{D}})$. Let $F=h+g$ and $\widetilde{v}=\mbox{Im}(F)$, where ``$\mbox{Im}$" is the imaginary part of a complex number.
Since $\mbox{Im}(iF)=\mbox{Im}(if)=u$, we see that
\be\label{eq-gj-0.1}\mathscr{M}_{\mbox{Re}(iF)}=\mathscr{M}_{\widetilde{v}}=\mathscr{M}_{u},\ee
where ``$\mbox{Re}$" is the real part of a complex number. By (\ref{eq-gj-0.1}) and Step 2, we have
$\widetilde{v}\in\Lambda_{\omega,p}(\overline{\mathbb{D}}),$ which implies that
\be\label{eq-gj-0.2}v_{1}=\frac{v+\widetilde{v}}{2}\in\Lambda_{\omega,p}(\overline{\mathbb{D}})\ee
and
\be\label{eq-gj-0.3}v_{2}=\frac{\widetilde{v}-v}{2}\in\Lambda_{\omega,p}(\overline{\mathbb{D}}).\ee
It follows from $\mathscr{M}_{u_{1}}=\mathscr{M}_{v_{1}}$, $\mathscr{M}_{u_{2}}=\mathscr{M}_{v_{2}}$ and Step 2  that
\be\label{eq-gj-0.4}u_{1}\in\Lambda_{\omega,p}(\overline{\mathbb{D}})\ee
and
\be\label{eq-gj-0.5}u_{2}\in\Lambda_{\omega,p}(\overline{\mathbb{D}}).\ee
Consequently, $h\in\Lambda_{\omega,p}(\overline{\mathbb{D}})$ follows from $(\ref{eq-gj-0.2})$ and $(\ref{eq-gj-0.4})$, and
$g\in\Lambda_{\omega,p}(\overline{\mathbb{D}})$ follows from $(\ref{eq-gj-0.3})$ and $(\ref{eq-gj-0.5})$.

\noindent $\mathbf{Step~4.}$ ``$(\mathscr{A}_{2})\Rightarrow(\mathscr{A}_{6})$",  ``$(\mathscr{A}_{2})\Rightarrow(\mathscr{A}_{8})$", ``$(\mathscr{A}_{3})\Rightarrow(\mathscr{A}_{2})$", ``$(\mathscr{A}_{3})\Rightarrow(\mathscr{A}_{4})$",
``$(\mathscr{A}_{2})\Rightarrow(\mathscr{A}_{5})$", ``$(\mathscr{A}_{6})\Rightarrow(\mathscr{A}_{7})$" and ``$(\mathscr{A}_{8})\Rightarrow(\mathscr{A}_{9})$"  are obvious.

\noindent $\mathbf{Step~5.}$ ``$(\mathscr{A}_{5})\Rightarrow(\mathscr{A}_{2})$".\\

Let $f$ be a harmonic $K$-quasiregular mapping in $\mathbb{D}$ with $|f|\in\Lambda_{\omega,p}(\overline{\mathbb{D}}).$
Then,  there is
a positive constant $M$ such that
\begin{equation}
\label{Lip-5}
\left(\int_{0}^{2\pi}||f(e^{i\eta}z_{1})|-|f(e^{i\eta}z_{2})||^{p}d\eta\right)^{\frac{1}{p}}\leq M\omega(|z_{1}-z_{2}|),
\end{equation}
for $z_1, z_2\in \overline{\mathbb{D}}$.
Let $z\in \mathbb{D}$ be fixed, and
let $\varepsilon_0=d_{\mathbb{D}}(z)/2$. Set $$\mathscr{N}_{p}=\int_{0}^{2\pi}\left(\mathscr{M}_{f}(ze^{i\eta})\right)^{p}d\eta.$$
Then, by Lemma \ref{mate-07d},
there is a positive constant $M^{\ast}$ such that
\begin{eqnarray*}
\mathscr{N}_{p}
&\leq &
\int_{0}^{2\pi}\frac{M^{\ast}}{\varepsilon_0^{p+2}}
\int_{\mathbb{D}(ze^{i\eta},\varepsilon_0)}\{|f(w)|-|f(ze^{i\eta})|\}_+^p dm(w)d\eta
\\
&=&
\int_{0}^{2\pi}\frac{M^{\ast}}{\varepsilon_0^{p+2}}
\int_{\mathbb{D}(z,\varepsilon_0)}\{|f(we^{i\eta})|-|f(ze^{i\eta})|\}_+^p dm(w)d\eta
\\
&=&
\frac{M^{\ast}}{\varepsilon_0^{p+2}}
\int_{\mathbb{D}(z,\varepsilon_0)}
\int_{0}^{2\pi}\{|f(we^{i\eta})|-|f(ze^{i\eta})|\}_+^p d\eta dm(w),
\end{eqnarray*}
which, together with (\ref{Lip-5}), implies that there is a positive constant $M$ such that
\beq\label{eq-kp-0.1}
\mathscr{N}_{p}&\leq&
\frac{M^{\ast}M^{p}}{\varepsilon_0^{p+2}}
\int_{\mathbb{D}(z,\varepsilon_0)}
\omega(|w-z|)^pdm(w)
\\ \nonumber
&\leq &
\frac{M^{\ast}M^{p}}{\varepsilon_0^{p}}
\big(\omega(\varepsilon_0)\big)^p
\\ \nonumber
&\leq &
2^pM^{\ast}M^{p}\frac{\big(\omega(d_{\mathbb{D}}(z))\big)^p}{\big(d_{\mathbb{D}}(z)\big)^{p}}.
\eeq
Combining (\ref{eq-kp-0.1}) and Step 2 gives
 $f\in\Lambda_{\omega,p}(\overline{\mathbb{D}}).$

\noindent $\mathbf{Step~6.}$  ``$(\mathscr{A}_{6})\Rightarrow(\mathscr{A}_{2})$". \\

Since $u\in\Lambda_{\omega,p}(\overline{\mathbb{D}}),$  by using the similar reasoning as in the proof of ``$(\mathscr{A}_{2})\Rightarrow(\mathscr{A}_{1})$",
we see that there is a positive constant $M$ such that, for $z\in\mathbb{D}$,
\be\label{h-j-0.1}\left(\int_{0}^{2\pi}\left(\mathscr{M}_{u}(ze^{i\theta})\right)^{p}d\theta\right)^{\frac{1}{p}}\leq M\frac{\omega\big(d_{\mathbb{D}}(z)\big)}{d_{\mathbb{D}}(z)}.\ee
Note that $u=(h+g+\overline{h+g})/2$, which gives that $\mathscr{M}_{u}=|h'+g'|$. Since $f$ is a harmonic $K$-quasiregular mapping in $\mathbb{D}$, we see that
\be\label{h-j-0.2}\frac{1}{K}\mathscr{M}_{f}\leq |h'|-|g'|\leq|h'+g'|=\mathscr{M}_{u}. \ee
Hence $(\mathscr{A}_{2})$ follows from (\ref{h-j-0.1}), (\ref{h-j-0.2}) and Step 2.

\noindent $\mathbf{Step~7.}$   ``$(\mathscr{A}_{8})\Rightarrow(\mathscr{A}_{2})$".\\

Since $v\in\Lambda_{\omega,p}(\overline{\mathbb{D}}),$  by using the similar reasoning as in the proof of ``$(\mathscr{A}_{2})\Rightarrow(\mathscr{A}_{1})$",
we see that there is a positive constant $M$ such that, for $z\in\mathbb{D}$,
\be\label{h-j-0.4}\left(\int_{0}^{2\pi}\left(\mathscr{M}_{v}(ze^{i\theta})\right)^{p}d\theta\right)^{\frac{1}{p}}\leq M\frac{\omega\big(d_{\mathbb{D}}(z)\big)}{d_{\mathbb{D}}(z)}.\ee
It is not difficult to know that $v=(h-g-\overline{(h-g)})/(2i)$, which yields that $\mathscr{M}_{v}=|h'-g'|$. Since $f$ is a harmonic $K$-quasiregular mapping in $\mathbb{D}$, we see that
\be\label{h-j-0.5}\frac{1}{K}\mathscr{M}_{f}\leq |h'|-|g'|\leq|h'-g'|=\mathscr{M}_{v}. \ee
Hence $(\mathscr{A}_{2})$ follows from (\ref{h-j-0.4}), (\ref{h-j-0.5}) and Step 2.

\noindent $\mathbf{Step~8.}$ ``$(\mathscr{A}_{4})\Rightarrow(\mathscr{A}_{3})$" follows from ``$(\mathscr{A}_{5})\Rightarrow(\mathscr{A}_{2})$".

\noindent $\mathbf{Step~9.}$  ``$(\mathscr{A}_{7})\Rightarrow(\mathscr{A}_{6})$" and
``$(\mathscr{A}_{9})\Rightarrow(\mathscr{A}_{8})$" follow from
Lemma \ref{real-harmonic-07d}
and the argument in ``$(\mathscr{A}_{5})\Rightarrow(\mathscr{A}_{2})$".

 \noindent $\mathbf{Step~10.}$   ``$(\mathscr{A}_{2})\Rightarrow(\mathscr{A}_{10})$".\\

 By using the similar reasoning as in the proof of (\ref{eq-0.3k}), we see that, for $e^{i\eta}\xi\in\mathbb{D}(e^{i\eta}z,d_{\mathbb{D}}(z)/2)$, there is a positive constant $M$
 such that
 $$\int_{0}^{2\pi}\left(\mathscr{M}_{f}(e^{i\eta}\xi)\right)^{p}d\eta
 \leq M^{p}\left(\frac{\omega\big(d_{\mathbb{D}}(z)\big)}{d_{\mathbb{D}}(z)}\right)^{p}.$$
 Then, by Lemma F, we have
 \beqq
 \mathscr{Y}_{u,q,p}(z)&\leq&\mathscr{Y}_{f,q,p}(z)
 \\
 &=&
 V_{q}(z)
 \left(\int_{0}^{2\pi}\left({\int_{\mathbb{D}\left(z,d_{\mathbb{D}}(z)/2\right)}\left(\mathscr{M}_{f}(\xi e^{i\eta})\right)^{q}dA(\xi)}
\right)^{\frac{p}{q}}d\eta\right)^{\frac{1}{p}}
\\
&\leq&
V_{q}(z)
\left(\int_{\mathbb{D}\left(z,d_{\mathbb{D}}(z)/2\right)}\left({\int_{0}^{2\pi}\left(\mathscr{M}_{f}(\xi e^{i\eta})\right)^{p}d\eta}
\right)^{\frac{q}{p}}dA(\xi)\right)^{\frac{1}{q}}
\\
&\leq&
V_{q}(z)
\left(\int_{\mathbb{D}\left(z,d_{\mathbb{D}}(z)/2\right)}\left(M^{p}\left(\frac{\omega\big(d_{\mathbb{D}}(z)\big)}{d_{\mathbb{D}}(z)}\right)^{p}
\right)^{\frac{q}{p}}dA(\xi)\right)^{\frac{1}{q}}
\\
&=&
 M\frac{\omega\big(d_{\mathbb{D}}(z)\big)}{d_{\mathbb{D}}(z)},
 \eeqq
where $ V_{q}(z)=1/\left|\mathbb{D}\left(z,d_{\mathbb{D}}(z)/2\right)\right|^{\frac{1}{q}}$ and $q\leq p<\infty$.

\noindent $\mathbf{Step~11.}$   ``$(\mathscr{A}_{10})\Rightarrow(\mathscr{A}_{2})$".\\

Since $\mathscr{M}_{u}=|h'+g'|$, we see that   $\mathscr{M}_{u}^{q}$ is subharmonic in $\mathbb{D}$
for $q\in (0,\infty)$. Then we have
 \beqq
 \big(\mathscr{M}_{u}(ze^{i\eta})\big)^{q}\leq\frac{1}{\left|\mathbb{D}\big(ze^{i\eta},\frac{d_{\mathbb{D}}(z)}{2}\big)\right|}
\int_{\mathbb{D}\big(ze^{i\eta},\frac{d_{\mathbb{D}}(z)}{2}\big)}\left(\mathscr{M}_{u}(\xi )\right)^{q}dA(\xi),
 \eeqq
 which, together with (\ref{h-j-0.2}) and $(\mathscr{A}_{10})$, implies  that there is a positive constant $M$ such that
 \beq\label{h-j-0.9}
 \frac{1}{K}\left(\int_{0}^{2\pi}\left(\mathscr{M}_{f}(ze^{i\eta})\right)^{p}d\eta\right)^{\frac{1}{p}}&\leq&
 \left(\int_{0}^{2\pi}\left(\mathscr{M}_{u}(ze^{i\eta})\right)^{p}d\eta\right)^{\frac{1}{p}}\\ \nonumber
 &\leq&\mathscr{Y}_{u,q,p}(z)\\ \nonumber
&\leq& M\frac{\omega\big(d_{\mathbb{D}}(z)\big)}{d_{\mathbb{D}}(z)},
 \eeq where $0<q\leq p<\infty$.
 It follows from (\ref{h-j-0.9}) and $(\mathscr{A}_{1})\Rightarrow(\mathscr{A}_{2})$ that $f\in\Lambda_{\omega,p}(\overline{\mathbb{D}}).$

\noindent $\mathbf{Step~12.}$  The proof of  $(\mathscr{A}_{11})\Leftrightarrow(\mathscr{A}_{2})$ is similar to  $(\mathscr{A}_{10})\Leftrightarrow(\mathscr{A}_{2})$.\\

$\mathbf{Case~2.}$  $p=\infty$. The technical route used to prove Case 1 is still valid for Case 2.
We only need to replace some formulas when $p\in[1,\infty)$ with corresponding formulas when $p=\infty$ to complete the proof. Therefore, we will not repeat it here.
The proof of this theorem is complete.
\qed

 \begin{Lem}\label{lem-4.0}
Let $\omega$ be a majorant  such that, for all $\delta\in[0,\pi]$,
\be\label{HJK-1}\delta\int_{\delta}^{\pi}\frac{\omega(t)}{t^{2}}dt\leq\,M\omega(\delta),\ee where $M$ is a positive constant.
For $p\in[1,\infty]$, if  $\varphi\in\Lambda_{\omega,p}(\mathbb{T})$, then there is a positive constant $M$
which depends only on $\| \varphi\|_{\Lambda_{\omega,p}(\mathbb{T}),s}$ such that
$$
 \left\{
 \begin{array}{ll}
 \left(\int_{0}^{2\pi}|\varphi\left(\widetilde{z}e^{i\eta}\right)-P[\varphi](ze^{i\eta})|^{p}d\eta\right)^{\frac{1}{p}}\leq M\omega(d_{\mathbb{D}}(z)), &p\in[1,\infty),
\\
|\varphi\left(\widetilde{z}\right)-P[\varphi](z)|\leq M\omega(d_{\mathbb{D}}(z)), &p=\infty.
\end{array}
\right.
$$
where $z\in\mathbb{D}\backslash\{0\}$ and $\widetilde{z}=z/|z|$.
\end{Lem}
\begin{proof}
Without loss of generality, we assume that $p\in[1,\infty)$.
For $z\in\mathbb{D}\setminus\{ 0\}$,
let $$J(z)=\left(\int_{0}^{2\pi}|\varphi\left(\widetilde{z}e^{i\eta}\right)-P[\varphi](ze^{i\eta})|^{p}d\eta\right)^{\frac{1}{p}}.$$
We divide the proof of this lemma into two cases.

\noindent $\mathbf{Case~1.}$ Let $z=re^{i\theta}\in\mathbb{D}$ with $r\in(0,1/4)$.
 Then
 \be\label{eq-re-0}\mathbf{P}(z,e^{i\tau})\leq\frac{1-r^{2}}{(1-r)^{2}}\leq\frac{5}{3}.\ee
It follows from $\varphi\in\Lambda_{\omega,p}(\mathbb{T})$, (\ref{eq-re-0}) and Lemma F
 that there is a positive constant $M$ which depends only on $\| \varphi\|_{\Lambda_{\omega,p}(\mathbb{T}),s}$
 such that
\beqq
J(z)&=&
\left(\int_{0}^{2\pi}\left|\int_{0}^{2\pi}\big(\varphi(e^{i(\theta+\eta)})-\varphi(e^{i(\tau+\eta)})\big)\mathbf{P}(z,e^{i\tau})\frac{d\tau}{2\pi}\right|^{p}
d\eta\right)^{\frac{1}{p}}\\
&\leq&\frac{5}{3}\left(\int_{0}^{2\pi}\left(\int_{0}^{2\pi}\big|\varphi(e^{i(\theta+\eta)})-\varphi(e^{i(\tau+\eta)})\big|\frac{d\tau}{2\pi}\right)^{p}
d\eta\right)^{\frac{1}{p}}\\
&\leq&\frac{5}{3}\int_{0}^{2\pi}\left(\int_{0}^{2\pi}\big|\varphi(e^{i(\theta+\eta)})-\varphi(e^{i(\tau+\eta)})\big|^{p}d\eta\right)^{\frac{1}{p}}
\frac{d\tau}{2\pi}\\
&\leq&\frac{5}{3}M\omega(2),
\eeqq
which implies that
\beqq\frac{J(z)}{\omega(d_{\mathbb{D}}(z))}\leq\frac{\frac{5}{3}M\omega(2)}{\omega\left(\frac{3}{4}\right)}. \eeqq

\noindent $\mathbf{Case~2.}$ Let $z=re^{i\theta}\in\mathbb{D}$ with $r\geq1/4$.

By Lemma F
and the assumption that $\varphi\in\Lambda_{\omega,p}(\mathbb{T})$, there is a positive constant $M$ which depends only on $\| \varphi\|_{\Lambda_{\omega,p}(\mathbb{T}),s}$
 such that
\beq\label{eq-rt-2.0}\nonumber
J(z)
&\leq&\left(\int_{0}^{2\pi}\left(\int_{0}^{2\pi}\big|\varphi(e^{i(\theta+\eta)})-\varphi(e^{i(\tau+\eta)})\big|\mathbf{P}(z,e^{i\tau})\frac{d\tau}{2\pi}\right)^{p}
d\eta\right)^{\frac{1}{p}}\\ \nonumber
&\leq&\int_{0}^{2\pi}\left(\int_{0}^{2\pi}\big|\varphi(e^{i(\theta+\eta)})-\varphi(e^{i(\tau+\eta)})\big|^{p}\left(\mathbf{P}(z,e^{i\tau})\right)^{p}d\eta\right)^{\frac{1}{p}}
\frac{d\tau}{2\pi}\\
&\leq&MJ_{1}(z),
\eeq
where
$$J_{1}(z)=\frac{1}{2\pi}\int_{0}^{2\pi}\mathbf{P}(z,e^{i\tau})\omega(|e^{i\tau}-e^{i\theta}|)d\tau.$$
Next, we estimate $J_{1}$.
Let $E_{1}(\theta)=\{\tau\in[-\pi+\theta, \pi+\theta]:~|\tau-\theta|\leq1-r\}$ and $E_{2}(\theta)=[-\pi+\theta, \pi+\theta]\backslash E_{1}(\theta)$.
Since $\sin x\geq2x/\pi $ for $x\in[0,\pi/2]$, we see that
\beqq
J_{1}(z)&=&\frac{1}{2\pi}
\int_{-\pi+\theta}^{\pi+\theta}\frac{1-r^{2}}{(1-r)^{2}+4r\left(\sin\frac{\theta-\tau}{2}\right)^{2}}\omega(|e^{i\tau}-e^{i\theta}|)d\tau\\
&\leq&\frac{1}{2\pi}\int_{E_{1}(\theta)\cup E_{2}(\theta)}\frac{1-r^{2}}{(1-r)^{2}+\frac{4}{\pi^{2}}r(\theta-\tau)^{2}}\omega(|\tau-\theta|)d\tau\\
&\leq&\frac{1}{2\pi}\int_{E_{1}(\theta)}\frac{1-r^{2}}{(1-r)^{2}}\omega(1-r)d\tau+\frac{\pi}{2}\int_{ E_{2}(\theta)}\frac{1-r^{2}}{(\theta-\tau)^{2}}\omega(|\tau-\theta|)d\tau\\
&\leq&\frac{2}{\pi}\omega(1-r)+2\pi(1-r)\int_{1-r}^{\pi}\frac{\omega(t)}{t^{2}}dt,
\eeqq
which, together with (\ref{HJK-1}) and (\ref{eq-rt-2.0}), gives
\beqq J(z)\leq\, M\left(2\pi+\frac{2}{\pi}\right)\omega(d_{\mathbb{D}}(z)).\eeqq

Combining Cases 1 and 2 gives the desired result.
\end{proof}

 \subsection{The proof of Theorem \ref{thm-5.0}} We  split the proof of theorem into two steps.

 \noindent $\mathbf{Step~1.}$ We first prove  ``$(\mathscr{C}_{2})\Rightarrow(\mathscr{C}_{1})$" for $p\in[1,\infty]$.

  Without loss of generality, we assume that $p\in[1,\infty)$.
  Also, it suffices to prove that $P[\varphi]\in\Lambda_{\omega,p}({\mathbb{D}})$.
  For $z_{1},~z_{2}\in{\mathbb{D}}$,
  let $$\mathscr{J}=\left(\int_{0}^{2\pi}\left|P[\varphi](z_{1}e^{i\eta})-P[\varphi](z_{2}e^{i\eta})\right|^{p}d\eta\right)^{\frac{1}{p}}.$$
  We  divide the proof of this step into two cases.

\noindent $\mathbf{Case~1.}$ Let $z_{1},~z_{2}\in\mathbb{D}$ with $\max\{|z_{1}|,~|z_{2}|\}\leq1/2$.

Since
\beqq
\mathbf{P}(z,e^{i\tau})=\frac{1-|z|^{2}}{|z-e^{i\tau}|^{2}}=\frac{e^{i\tau}}{e^{i\tau}-z}+\frac{\overline{z}}{e^{-i\tau}-\overline{z}},
\eeqq
we see that,
for $z\in\mathbb{D}$  with $|z|\leq1/2$,
\beqq
\mathscr{M}_{\mathbf{P}}(z)=\frac{1}{|e^{i\tau}-z|^{2}}+\frac{1}{|e^{-i\tau}-\overline{z}|^{2}}\leq8.
\eeqq
Consequently,
\be\label{eq-rt-2.2}
\left|\varpi_{\tau}(z_{1},z_{2})\right|\leq\max_{|z|\leq1/2}\{\mathscr{M}_{\mathbf{P}}(z)\}|z_{1}-z_{2}|
\leq\frac{16}{\omega(2)}\omega(|z_{1}-z_{2}|),
\ee where $\varpi_{\tau}(z_{1},z_{2})=\mathbf{P}(z_{1},e^{i\tau})-\mathbf{P}(z_{2},e^{i\tau}).$

It follows from  (\ref{eq-rt-2.2}), $\varphi\in\Lambda_{\omega,p}(\mathbb{T})$ and Lemma F
 that there is a positive constant $M$ such that
\beqq
\mathscr{J}&=&\left(\int_{0}^{2\pi}\left|\int_{0}^{2\pi}\big(\varphi(e^{i(\tau+\eta)})-\varphi(e^{i\eta})\big)
(\varpi_{\tau}(z_{1},z_{2}))\frac{d\tau}{2\pi}\right|^{p}
d\eta\right)^{\frac{1}{p}}\\
&\leq&\int_{0}^{2\pi}\left(\int_{0}^{2\pi}|\varphi(e^{i(\tau+\eta)})-\varphi(e^{i\eta})|^{p}
\left|\varpi_{\tau}(z_{1},z_{2})\right|^{p}d\eta\right)^{\frac{1}{p}}\frac{d\tau}{2\pi}\\
&\leq&\frac{16}{\omega(2)}\omega(|z_{1}-z_{2}|)\int_{0}^{2\pi}\left(\int_{0}^{2\pi}|\varphi(e^{i(\tau+\eta)})-\varphi(e^{i\eta})|^{p}
d\eta\right)^{\frac{1}{p}}\frac{d\tau}{2\pi}\\
&\leq&16M\omega(|z_{1}-z_{2}|).
\eeqq

\noindent $\mathbf{Case~2.}$ Let $z_{1},~z_{2}\in\mathbb{D}$ with $\max\{|z_{1}|,~|z_{2}|\}>1/2$.
In this case, we write $z_{j}=r_{j}e^{i\theta_{j}}$  for  $j\in\{1,2\}$, where $r_{j}=|z_{j}|$.
Without loss of generality, we assume that $r_{2}\leq r_{1}$. For any fixed $r\in(0,1)$, let
$$F_{r}(z)=P[\varphi](rz),~z\in\overline{\mathbb{D}}.$$
Moreover, for $\xi_{1},~\xi_{2}\in\mathbb{T}$, let
\beqq
\mathscr{J}_{1}=\left(\int_{0}^{2\pi}|F_{r}(e^{i\eta}\xi_{1})-F_{r}(e^{i\eta}\xi_{2})|^{p}d\eta\right)^{\frac{1}{p}}.
\eeqq
Then, by $\varphi\in\Lambda_{\omega,p}(\mathbb{T})$ and Lemma F, we see that there is a positive constant $M$ which is independent of $r$ such that
\beq\label{eq-rt-2.4}\nonumber
\mathscr{J}_{1}&\leq&\int_{0}^{2\pi}\mathbf{P}(r,e^{i\tau})
\left(\int_{0}^{2\pi}\left|\varphi(e^{i(\tau+\eta)}\xi_{1})-\varphi(e^{i(\tau+\eta)}\xi_{2})\right|^{p}
d\eta\right)^{\frac{1}{p}}
\frac{d\tau}{2\pi}\\
&\leq&M\omega(|\xi_{1}-\xi_{2}|),
\eeq which implies that $F_{r}|_{\mathbb{T}}\in\Lambda_{\omega,p}(\mathbb{T})$
and $\| F_{r}|_{\mathbb{T}}\|_{\Lambda_{\omega,p}(\mathbb{T}),s}$
is independent of $r\in (0,1)$.

By Minkowski's inequality, we have
\beqq
\mathscr{J}&=&\bigg(\int_{0}^{2\pi}\big|F_{r_{1}}(e^{i(\theta_{1}+\eta)})-F_{r_{1}}(e^{i(\theta_{2}+\eta)})\\
&&+
F_{r_{1}}(e^{i(\theta_{2}+\eta)})-F_{r_{2}}(e^{i(\theta_{2}+\eta)})\big|^{p}d\eta\bigg)^{\frac{1}{p}}\\
&\leq&\mathscr{J}_{2}+\mathscr{J}_{3},
\eeqq
where $$\mathscr{J}_{2}=\left(\int_{0}^{2\pi}\left|F_{r_{1}}(e^{i(\theta_{1}+\eta)})-F_{r_{1}}(e^{i(\theta_{2}+\eta)})\right|^{p}d\eta\right)^{\frac{1}{p}}$$
and
$$\mathscr{J}_{3}=\left(\int_{0}^{2\pi}\left|F_{r_{1}}(e^{i(\theta_{2}+\eta)})-F_{r_{2}}(e^{i(\theta_{2}+\eta)})\right|^{p}d\eta\right)^{\frac{1}{p}}.$$
By (\ref{eq-rt-2.4}), we see that there is a positive constant $M$ such that
\be\label{eq-rt-2.5}
\mathscr{J}_{2}\leq M\omega(|e^{i\theta_{1}}-e^{i\theta_{2}}|).
\ee
Since, for any fixed $\eta\in[0,2\pi]$,
$$F_{r_{2}}(e^{i(\theta_{2}+\eta)})=F_{r_{1}}\left(\frac{r_{2}}{r_{1}}e^{i(\theta_{2}+\eta)}\right)=
P[F_{r_{1}}|_{\mathbb{T}}]\left(\frac{r_{2}}{r_{1}}e^{i(\theta_{2}+\eta)}\right),$$
by $F_{r_{1}}|_{\mathbb{T}}\in\Lambda_{\omega,p}(\mathbb{T})$,
$\| F_{r_1}|_{\mathbb{T}}\|_{\Lambda_{\omega,p}(\mathbb{T}),s}$
is independent of $r_1\in (0,1)$ and Lemma \ref{lem-4.0}, we see that
there is a positive constant $M$ such that
\be\label{eq-rt-2.6}
\mathscr{J}_{3}\leq M\omega\left(1-\frac{r_{2}}{r_{1}}\right).
\ee
From $r_{1}>1/2$, we obtain that
$$|e^{i\theta_{1}}-e^{i\theta_{2}}|\leq4|z_{1}-z_{2}|~\mbox{and}~1-\frac{r_{2}}{r_{1}}\leq2|z_{1}-z_{2}|,$$
which, together with (\ref{eq-rt-2.5}) and (\ref{eq-rt-2.6}), implies that there
is a positive constant $M$ such that
\[
\mathscr{J}\leq M\omega(|z_{1}-z_{2}|).
\]
Hence $(\mathscr{C}_{1})$ follows from Cases 1 and 2 for $p\in[1,\infty)$.

\noindent $\mathbf{Step~2.}$   ``$(\mathscr{C}_{1})\Rightarrow(\mathscr{C}_{2})$" for $p=\infty$.

Let
\be\label{bjh-1}
\varphi(e^{i\tau})=
\left\{
\begin{array}{ll}
\omega(\tau), & \tau\in[0,\pi], \\
\omega(2\pi-\tau), & \tau\in[\pi,2\pi].
\end{array}
\right.
\ee
Since
\be\label{lk-0.1}\omega(t_{1}+t_{2})\leq\omega(t_{1})+\omega(t_{2})\ee
for $t_{1},~t_{2}\in[0,\infty)$, we see that
$\varphi\in\Lambda_{\omega,\infty}(\mathbb{T})$.  Then, by $(\mathscr{C}_{1})$, we have $P[\varphi]\in\Lambda_{\omega,\infty}(\overline{\mathbb{D}})$.
In the following, we split the remaining proof of this step into two cases.

\noindent $\mathbf{Case~3.}$ Let $\delta\in[0,1/2]$.

In this case, let $r=1-\delta$. Since
$$P[\varphi](1)=\varphi(1)=\omega(0)=0,$$ by $P[\varphi]\in\Lambda_{\omega,\infty}(\overline{\mathbb{D}})$,
 we see that there is a positive constant $M$ such that
\beqq
\left|P[\varphi](r)\right|=\left|P[\varphi](r)-P[\varphi](1)\right|\leq M\omega(1-r),
\eeqq
which gives that
\beqq
\mathscr{J}_{4}&\leq&2\int_{1-r}^{\pi}\frac{(1-r)\omega(t)}{(1-r)^{2}+t^{2}}dt
\leq 2\int_{1-r}^{\pi}\frac{(1-r^{2})\omega(t)dt}{(1-r)^{2}+4r\sin^{2}\left(\frac{t}{2}\right)}\\ \nonumber
&\leq& 2\left|\int_{0}^{\pi}\mathbf{P}(r,e^{i\tau})\omega(\tau)d\tau
+\int_{\pi}^{2\pi}\mathbf{P}(r,e^{i\tau})\omega(2\pi-\tau)d\tau\right|\\ \nonumber
&=& 4\pi \left|P[\varphi](r)\right|\\
&\leq& 4\pi M\omega(1-r),
\eeqq where $$\mathscr{J}_{4}=(1-r)\int_{1-r}^{\pi}\frac{\omega(t)}{t^{2}}dt.$$

\noindent $\mathbf{Case~4.}$ Let $\delta\in[1/2, \pi]$.

For this case, we have
$$
\delta\int_{\delta}^{\pi}\frac{\omega(t)}{t^{2}}dt\leq\pi\int_{\frac{1}{2}}^{\pi}\frac{\omega(t)}{t^{2}}dt\leq
2\pi\int_{\frac{1}{2}}^{\pi}\frac{\omega(t)}{t}dt\leq2\pi^{2}\frac{\omega\left(\frac{1}{2}\right)}{\frac{1}{2}}\leq4\pi^{2}\omega(\delta).
$$
Combining Cases 3 and 4 gives $(\mathscr{C}_{2})$.
 The proof of this theorem is complete.
\qed



 \subsection{The proof of Theorem \ref{t-3}} $\mathbf{Case~1.}$ We first give a proof in the case $p\in[1,\infty)$.
We divide the proof of this case into two steps.

\noindent $\mathbf{Step~1.}$ ``$(\mathscr{D}_{1})\Rightarrow(\mathscr{D}_{2})$".\\

Since $f$ is continuous up to the boundary,
$f\in\Lambda_{\omega,p}(\mathbb{D})$ implies that
\be\label{eq-jl-1}|f|\in\Lambda_{\omega,p}(\overline{\mathbb{D}}).\ee
Since $|f|\in\Lambda_{\omega,p}(\mathbb{T})$
and $\omega$ is a regular majorant,
by Theorem \ref{thm-5.0} (or Corollary \ref{cor-0.11}), we see that
\be\label{eq-jl-2}P[|f|]\in\Lambda_{\omega,p}(\overline{\mathbb{D}}).\ee
Let $$\mathscr{J}_{5}=\left(\int_{0}^{2\pi}\left(P[|f|](ze^{i\theta})-|f(ze^{i\theta})|\right)^{p}d\theta\right)^{\frac{1}{p}}.$$
Then, by (\ref{eq-jl-1}), (\ref{eq-jl-2}) and  the Minkowski inequality, we see that there is a positive constant $M$ such that
\beqq
\mathscr{J}_{5}
&\leq&
\left(\int_{0}^{2\pi}\left|P[|f|](ze^{i\theta})-|f(\widetilde{z}e^{i\theta})|\right|^{p}d\theta\right)^{\frac{1}{p}}\\
&&+\left(\int_{0}^{2\pi}\left||f(\widetilde{z}e^{i\theta})|-|f(ze^{i\theta})|\right|^{p}d\theta\right)^{\frac{1}{p}}\\
&\leq&M\omega(d_{\mathbb{D}}(z)),
\eeqq where $z\in\mathbb{D}\backslash\{0\}$ and $\widetilde{z}=z/|z|$. Hence $(\mathscr{D}_{2})$ holds.

\noindent $\mathbf{Step~2.}$ ``$(\mathscr{D}_{2})\Rightarrow(\mathscr{D}_{1})$".\\

For a fixed point $z\in\mathbb{D}$ and a fixed point $\theta\in[0,2\pi]$, we have
\beq\label{kkl-1}
\{|f(we^{i\theta})|-|f(ze^{i\theta})|\}_{+}&\leq&\{P[|f|](we^{i\theta})-|f(ze^{i\theta})|\}_{+}\\ \nonumber
&\leq&\{P[|f|](we^{i\theta})-P[|f|](ze^{i\theta})\}_{+}\\ \nonumber
&&+\{P[|f|](ze^{i\theta})-|f(ze^{i\theta})|\}_{+},
\eeq
where $w\in\{\varsigma:~|\varsigma-z|\leq d_{\mathbb{D}}(z)\}.$

From $(\mathscr{D}_{2})$, we know that there is a positive constant $M$ such that
\be \label{kkl-2} \left(\int_{0}^{2\pi}\left(P[|f|](ze^{i\theta})-|f(ze^{i\theta})|\right)^{p}d\theta\right)^{\frac{1}{p}}
\leq M\omega(d_{\mathbb{D}}(z)).\ee
Since $|f|\in\Lambda_{\omega,p}(\mathbb{T})$
and $\omega$ is a regular majorant,
by Theorem \ref{thm-5.0} (or Corollary \ref{cor-0.11}), we see that
that there is a positive constant $M$ such that
\be \label{kkl-3} \left(\int_{0}^{2\pi}\left|P[|f|](we^{i\theta})-P[|f|](ze^{i\theta})\right|^{p}d\theta\right)^{\frac{1}{p}}
\leq M\omega(d_{\mathbb{D}}(z)),\ee
for $w\in\{\varsigma:~|\varsigma-z|\leq d_{\mathbb{D}}(z)\}.$
Combining (\ref{kkl-1}), (\ref{kkl-2}), (\ref{kkl-3}) and the Minkowski inequality gives that
\beqq
\left(\int_{0}^{2\pi}\left\{|f(we^{i\theta})|-|f(ze^{i\theta})|\right\}_{+}^{p}d\theta\right)^{\frac{1}{p}}
\leq 2M\omega(d_{\mathbb{D}}(z)),
\eeqq
for $w\in\{\varsigma:~|\varsigma-z|\leq d_{\mathbb{D}}(z)\}.$ Arguing as in the proof of the Step 5 of Theorem \ref{thm-1.0},
we  have $f\in\Lambda_{\omega,p}(\mathbb{D})$.

$\mathbf{Case~2.}$  $p=\infty$. The proof method used to prove Case 1 is still valid for Case 2.
We only need to replace some formulas when $p\in[1,\infty)$ with corresponding formulas when $p=\infty$ to complete the proof. Therefore, we omit it here.

The proof of this theorem is finished.
\qed

 \subsection{The proof of Theorem \ref{thm-8}} We only need to prove the case $p\in[2,\infty)$ because the case $p=\infty$ follows from Theorem C.
  We first prove the necessity. For $z\in\mathbb{D}$,
let
$$\mathscr{S}_{p}(z)=\left(\int_{0}^{2\pi}\left(P[|f|^{2}](ze^{i\eta})-|f(ze^{i\eta})|^{2}\right)^{\frac{p}{2}}d\eta\right)^{\frac{1}{p}}.$$
Since $f\in\Lambda_{\omega,p}(\overline{\mathbb{D}})$, we see that there is a positive constant $M$ such that
\be\label{xx-1}
\int_{0}^{2\pi}|f(e^{i(\tau+\eta)})-f(ze^{i\eta})|^{p}d\eta\leq M^{p}\big(\omega(|z-e^{i\tau}|)\big)^{p}.
\ee
Then, by (\ref{xx-1}) and Lemma F, there is a positive constant $M$ such that
\beq \nonumber
\mathscr{S}_{p}(z)&=&\left(\int_{0}^{2\pi}\left(\int_{0}^{2\pi}|f(e^{i(\tau+\eta)})-f(ze^{i\eta})|^{2}\mathbf{P}(z,e^{i\tau})
\frac{d\tau}{2\pi}\right)^{\frac{p}{2}}d\eta\right)^{\frac{1}{p}}\\ \nonumber
&\leq&\left(\int_{0}^{2\pi}\mathbf{P}(z,e^{i\tau})\left(\int_{0}^{2\pi}|f(e^{i(\tau+\eta)})-f(ze^{i\eta})|^{p}d\eta\right)^{\frac{2}{p}}\frac{d\tau}{2\pi}\right)^{\frac{1}{2}}\\
&\leq&M\left(\int_{0}^{2\pi}\mathbf{P}(z,e^{i\tau})\big(\omega(|z-e^{i\tau}|)\big)^{2}\frac{d\tau}{2\pi}\right)^{\frac{1}{2}}.
\label{xx-2}
\eeq
On the other hand, for $z\neq0$ and $\widetilde{z}=z/|z|$, it follows from (\ref{lk-0.1}) that
\beq\label{xx-3}
\big(\omega(|z-e^{i\tau}|)\big)^{2}&\leq&\left(\omega(|\widetilde{z}-e^{i\tau}|)+\omega(|z-\widetilde{z}|)\right)^{2}\\ \nonumber
&\leq&2\left(\big(\omega(|\widetilde{z}-e^{i\tau}|)\big)^{2}+\big(\omega(d_{\mathbb{D}}(z))\big)^{2}\right).
\eeq
Combining (\ref{xx-2}), (\ref{xx-3}) and \cite[Lemma 2]{Dy1}
yields that
there is a positive constant $M$
such that $$ \left(\int_{0}^{2\pi}\left(P[|f|^{2}](ze^{i\theta})-|f(ze^{i\theta})|^{2}\right)^{\frac{p}{2}}d\theta\right)^{\frac{1}{p}}
\leq M\omega(d_{\mathbb{D}}(z)).$$

Next, we prove the sufficiency. For $z\in\mathbb{D}$ and fixed $\eta\in[0,2\pi]$, it follows from the Cauchy integral formula and the Cauchy-Schwarz inequality that
\beq\label{xx-4} \nonumber
d_{\mathbb{D}}(z)|f'(ze^{i\eta})|&=&\left|\int_{|\zeta|=1}\left(f(\zeta e^{i\eta})-f(ze^{i\eta})\right)\frac{d_{\mathbb{D}}(z)}{(\zeta-z)^{2}}\frac{d\zeta}{2\pi i}\right|\\ \nonumber
&\leq&\int_{0}^{2\pi}|f(\zeta e^{i\eta})-f(ze^{i\eta})|\mathbf{P}(z,e^{i\tau})\frac{d\tau}{2\pi}\\
&\leq&\left(\int_{0}^{2\pi}|f(\zeta e^{i\eta})-f(ze^{i\eta})|^{2}\mathbf{P}(z,e^{i\tau})\frac{d\tau}{2\pi}\right)^{\frac{1}{2}},
\eeq where $\zeta=e^{i\tau}$.
For $z\in\mathbb{D}$, let $$\mathscr{Q}_{p}(z)=d_{\mathbb{D}}(z)\left(\int_{0}^{2\pi}|f'(ze^{i\eta})|^{p}d\eta\right)^{\frac{1}{p}}.$$
Then, by (\ref{xx-2}) and (\ref{xx-4}), we see that there is a positive constant $M$ such that
\beqq\mathscr{Q}_{p}(z)
&\leq&
\left(\int_{0}^{2\pi}\left(\int_{0}^{2\pi}|f(e^{i(\tau+\eta)})-f(ze^{i\eta})|^{2}\mathbf{P}(z,e^{i\tau})\frac{d\tau}{2\pi}\right)^{\frac{p}{2}}d\eta\right)^{\frac{1}{p}}\\
&=&\mathscr{S}_{p}(z)\\
&\leq&M\omega(d_{\mathbb{D}}(z)),
\eeqq
which, together with Corollary \ref{An-th-0.1}, gives that $f\in\Lambda_{\omega,p}(\overline{\mathbb{D}})$. The proof of this theorem is complete.
\qed

 \section{Applications of equivalent norms and Hardy-Littlewood type theorems}\label{sec3}

A continuous non-decreasing function $\psi:~[0,1)\rightarrow(0,\infty)$ is called a
 weight if $\psi$ is  unbounded (see \cite{AD}). Moreover,
 a weight  $\psi$ is called doubling if there is a constant $M>1$
such that
$$\psi(1-s/2)< M\psi(1-s)$$
for $s\in(0,1]$.
The following result easily follows from \cite[Lemma 1]{AD} and
\cite[Theorem 2]{AD}.

\begin{Lem}\label{lem-0.2}
Let  $\psi$ be a doubling function. Then there exist holomorphic  functions
$f_{j}~(j\in\{1,2\})$
with
\[
\sup_{z\in \mathbb{D}}\frac{|f_j'(z)|}{\psi(|z|)}<\infty
\]
such that
for $z\in\mathbb{D}$,
\[
\sum_{j=1}^{2}|f_{j}'(z)|\geq\psi(|z|).
\]
\end{Lem}

\begin{LemG}{\rm (\cite[Lemma 2]{P})}\label{L-3}
Suppose that $f\in \mathscr{A}(\mathbb{D})$. For $z\in\mathbb{D}$, let $D_{z}=\{w:~|w-z|\leq d_{\mathbb{D}}(z)\}$ and
$M_{z}=\max\{|f(w)|:~w\in D_{z}\}$. Then, for $z\in\mathbb{D}$,
$$\frac{1}{2}(1-|z|)|f'(z)|+|f(z)|\leq\,M_{z}.$$
\end{LemG}

\begin{Lem}\label{lem-0.4}
Let $\omega_{1}$ be a majorant such that $\omega_{1}$ is differentiable on $(0,1
]$ and $\omega_{1}'$ is also non-increasing on $(0,1
]$, and let  $\omega_{2}$ be a fast majorant. For $p\in[1,\infty]$, if $\phi$
is a holomorphic function of $\mathbb{D}$ into itself,
then
$\omega_{1}(d_{\mathbb{D}}(\phi))\in\Lambda_{\omega_{2},p}(\mathbb{D})$
if and only if there is a positive constant $M$ such that
for $z\in\mathbb{D}$,
\be\label{lem-3j}
\left\{
 \begin{array}{ll}
 \left(\int_{0}^{2\pi}\left(\mathscr{M}_{\omega_{1}(d_{\mathbb{D}}(\phi))}(ze^{i\theta})\right)^{p}d\theta\right)^{\frac{1}{p}}
\leq\,M\frac{\omega_{2}(d_{\mathbb{D}}(z))}{d_{\mathbb{D}}(z)}, &p\in[1,\infty),
\\
\mathscr{M}_{\omega_{1}(d_{\mathbb{D}}(\phi))}(z)
\leq\,M\frac{\omega_{2}(d_{\mathbb{D}}(z))}{d_{\mathbb{D}}(z)}, &p=\infty.
\end{array}
\right.
\ee

\end{Lem}

 \begin{proof}
 Without loss of generality,  we assume that $\phi$ is non-constant and $p\in[1,\infty)$.
We first prove the sufficiency. For $z\in\mathbb{D}$, elementary calculations give
$$
\left|\frac{\partial\omega_{1}(d_{\mathbb{D}}(\phi(z)))}{\partial z}\right|=
\left|\frac{\partial d_{\mathbb{D}}(\phi(z))}{\partial z}\right|\omega_{1}'(d_{\mathbb{D}}(\phi(z)))=\frac{|\phi'(z)|\omega_{1}'(d_{\mathbb{D}}(\phi(z)))}{2}
$$
and
$$
\left|\frac{\partial\omega_{1}(d_{\mathbb{D}}(\phi(z)))}{\partial \overline{z}}\right|=
\left|\frac{\partial d_{\mathbb{D}}(\phi(z))}{\partial \overline{z}}\right|\omega_{1}'(d_{\mathbb{D}}(\phi(z)))=\frac{|\phi'(z)|\omega_{1}'(d_{\mathbb{D}}(\phi(z)))}{2},
$$
which gives that
\be\label{eq-t6}
\mathscr{M}_{\omega_{1}(d_{\mathbb{D}}(\phi))}(z)=|\phi'(z)|\omega_{1}'
(d_{\mathbb{D}}(\phi(z))),
\quad
z\in \mathbb{D}.
\ee
Since $\mathbb{D}$ is a $\Lambda_{\omega_{2}}$-extension domain of $\mathbb{C}$ for fast majorant $\omega_{2}$,
we see that, for $z,~w\in\mathbb{D}$, there is a rectifiable curve $\gamma\subset\mathbb{D}$ joining $z$ to $w$ such that
\be\label{eq-t7}\int_{\gamma}\frac{\omega_{2}(d_{\mathbb{D}}(\zeta))}{d_{\mathbb{D}}(\zeta)}ds(\zeta)\leq M\omega_{2}(|z-w|),\ee
where $M$ is a positive constant.

By (\ref{lem-3j}), (\ref{eq-t6}), (\ref{eq-t7}) and Lemma F,  we conclude that, for $z,~w\in\mathbb{D}$, there is a rectifiable
curve $\gamma\subset\mathbb{D}$ joining $z$ to $w$ such that
\beqq
I(z,w)&\leq&
\left(\int_{0}^{2\pi}\left(\int_{\gamma}\mathscr{M}_{\omega_{1}(d_{\mathbb{D}}(\phi))}(e^{i\theta}\zeta)ds(\zeta)\right)^{p}d\theta\right)^{\frac{1}{p}}\\
&\leq&\int_{\gamma}\left(\int_{0}^{2\pi}\left(\mathscr{M}_{\omega_{1}(d_{\mathbb{D}}(\phi))}(\zeta e^{i\theta})\right)^{p}d\theta\right)^{\frac{1}{p}}ds(\zeta)\\
&\leq&M\int_{\gamma}\frac{\omega_{2}(d_{\mathbb{D}}(\zeta))}{d_{\mathbb{D}}(\zeta)}ds(\zeta)\\
&\leq&M^{2}\omega_{2}(|z-w|),
\eeqq
where $$I(z,w)=\left(\int_{0}^{2\pi}\left|\omega_{1}(d_{\mathbb{D}}(\phi(ze^{i\theta})))-
\omega_{1}(d_{\mathbb{D}}(\phi(we^{i\theta})))\right|^{p}d\theta\right)^{\frac{1}{p}}.$$

Next, we prove the necessity. For $a,b\in [0,1]$ with $b>a$, it follows from the Lagrange mean value theorem that there is an $a_{0}\in(a,b)$
such that
\be\label{eq-t4} \frac{\omega_{1}(a)-\omega_{1}(b)}{a-b}=\omega_{1}'(a_{0})\geq\omega_{1}'(b).\ee
Since $\omega_{1}(d_{\mathbb{D}}(\phi))\in\Lambda_{\omega_{2},p}(\mathbb{D})$, we see that
there is a positive constant $M$ such that
\be\label{eq-t1}\mathscr{J}_{6}
\leq\,M\omega_{2}(|z-w|)
\ee for $z,w\in{\mathbb{D}}$,
where $$\mathscr{J}_{6}=\left(\int_{0}^{2\pi}\left|\omega_{1}(d_{\mathbb{D}}(\phi(ze^{i\theta})))-
\omega_{1}(d_{\mathbb{D}}(\phi(we^{i\theta})))\right|^{p}d\theta\right)^{\frac{1}{p}}.$$
By  (\ref{eq-t4}) and (\ref{eq-t1}), we have
\[
\mathscr{J}_{7}
\leq\,M\omega_{2}(|z-w|),~
\]
for $z,w\in{\mathbb{D}}$,
where $$\mathscr{J}_{7}=\left(\int_{0}^{2\pi}\left(\left\{
|\phi(we^{i\theta})|-|\phi(ze^{i\theta})|
\right\}_{+}\omega_{1}'(d_{\mathbb{D}}(\phi(ze^{i\theta})))\right)^{p}d\theta\right)^{\frac{1}{p}}.$$

Then, for $z,w\in {\mathbb{D}}$ with
$|w-z|\leq d_{\mathbb{D}}(z)/2$,
we have
\be\label{eq-t2b}
\mathscr{J}_{7}
\leq\,M\omega_{2}(d_{\mathbb{D}}(z)).
\ee
Let $$\mathscr{W}(z)=\int_{0}^{2\pi}\left(\mathscr{M}_{\omega_{1}(d_{\mathbb{D}}(\phi))}(ze^{i\theta})\right)^{p}d\theta$$
and
$$\mathscr{J}_{8}(w,ze^{i\theta})=\{|\phi(w)|-|\phi(ze^{i\theta})|\}_+ \omega_{1}'(d_{\mathbb{D}}(\phi(ze^{i\theta}))).$$
For a fixed $z\in \mathbb{D}$  and
 $\varepsilon_0=d_{\mathbb{D}}(z)/2$,
it follows from Lemma \ref{mate-07d} that
there is a positive constant $C$ such that
\begin{eqnarray*}
\mathscr{W}(z)
&\leq&
\int_{0}^{2\pi}\frac{C}{\varepsilon_0^{p+2}}
\int_{\mathbb{D}(ze^{i\theta},\varepsilon_0)}\left(\mathscr{J}_{8}(w,ze^{i\theta})\right)^p dm(w)d\theta
\\
&=&
\int_{0}^{2\pi}\frac{C}{\varepsilon_0^{p+2}}
\int_{\mathbb{D}(z,\varepsilon_0)}\left(\mathscr{J}_{8}(we^{i\theta},ze^{i\theta})\right)^p dm(w)d\theta
\\
&=&
\frac{C}{\varepsilon_0^{p+2}}
\int_{\mathbb{D}(z,\varepsilon_0)}
\int_{0}^{2\pi}\left(\mathscr{J}_{8}(we^{i\theta},ze^{i\theta})\right)^p d\theta dm(w),
\end{eqnarray*}
which, together with (\ref{eq-t2b}), yields that
\beqq
\mathscr{W}(z)&\leq&
\frac{C}{\varepsilon_0^{p+2}}
\int_{\mathbb{D}(z,\varepsilon_0)}
M^p\omega_{2}(d_{\mathbb{D}}(z))^pdm(w)
\\
&\leq&
\frac{C}{\varepsilon_0^{p}}
M^p\big(\omega_{2}(d_{\mathbb{D}}(z))\big)^p
\\
&\leq&2^pCM^p
\frac{\big(\omega_{2}(d_{\mathbb{D}}(z))\big)^p}{\big(d_{\mathbb{D}}(z)\big)^{p}}.
\eeqq
The proof of this lemma is complete.
 \end{proof}

 From the  necessity proof of Lemma \ref{lem-0.4}, we also  obtain the following result.
 Here we only need to assume that $\omega_{2}$ is a  majorant because the ``fast" condition only used in
 the sufficiency proof of Lemma \ref{lem-0.4}.

\begin{Lem}\label{cor-1.0}
Let $\omega_{1}$ be a majorant such that $\omega_{1}$ is differentiable on $(0,1
]$ and $\omega_{1}'$ is also non-increasing on $(0,1
]$, and let  $\omega_{2}$ be a  majorant.
Suppose that  $\phi$
is a holomorphic function of $\mathbb{D}$ into itself.
For $p\in[1,\infty]$, if
$$
\left\{
 \begin{array}{ll}
\chi_{\omega_{1}}(z,w)
\leq\omega_{2}(d_{\mathbb{D}}(z)), &p\in[1,\infty),
\\
\left\{\omega_{1}(d_{\mathbb{D}}(\phi(z)))-
\omega_{1}(d_{\mathbb{D}}(\phi(w)))\right\}_{+}
\leq\omega_{2}(d_{\mathbb{D}}(z)), &p=\infty.
\end{array}
\right.
$$ whenever $z\in\mathbb{D}$ and $w\in\{\varsigma\in\mathbb{D}:~|\varsigma-z|\leq\,d_{\mathbb{D}}(z)/2\}$,
where $$\chi_{\omega_{1}}(z,w)=\left(\int_{0}^{2\pi}\left\{\omega_{1}(d_{\mathbb{D}}(\phi(ze^{i\theta})))-
\omega_{1}(d_{\mathbb{D}}(\phi(we^{i\theta})))\right\}_{+}^{p}d\theta\right)^{\frac{1}{p}},$$
then there is a positive constant $M$ such that
for $z\in \mathbb{D}$,
$$
\left\{
 \begin{array}{ll}
\left(\int_{0}^{2\pi}\left(\mathscr{M}_{\omega_{1}(d_{\mathbb{D}}(\phi))}(ze^{i\theta})\right)^{p}d\theta\right)^{\frac{1}{p}}
\leq\,M\frac{\omega_{2}(d_{\mathbb{D}}(z))}{d_{\mathbb{D}}(z)}, &p\in[1,\infty),
\\
\mathscr{M}_{\omega_{1}(d_{\mathbb{D}}(\phi))}(z)
\leq\,M\frac{\omega_{2}(d_{\mathbb{D}}(z))}{d_{\mathbb{D}}(z)}, &p=\infty.
\end{array}
\right.
$$
\end{Lem}


 \begin{Lem}\label{le-9}
Suppose that $\omega$ is a majorant such that $\omega$ is differentiable on $(0,1
]$ and $\omega'$ is also non-increasing on $(0,1
]$.
Let $\phi$ be an analytic function of $\mathbb{D}$ into itself.
Then $\omega(d_{\mathbb{D}}(\phi))$ is superharmonic in $\mathbb{D}$.
\end{Lem}

 \begin{proof}
  Without loss of generality,
we assume that $\phi$ is non-constant.
Since $\phi$ is analytic in $\mathbb{D}$, we see that
$|\phi(z)|$ is subharmonic in $\mathbb{D}$.
Let $\varphi(t)=-\omega(1-t)$ for $t\in (-\infty,1]$.
Then, $\varphi'(t)$ exists and is non-decreasing on $[0,1)$.
Let $t_0\in [0,1)$ be arbitrarily fixed.
For $0\leq t<t_0< 1$, it follows from the Lagrange mean value theorem that there is an $s\in(t,t_0)$
such that
\[
\frac{\varphi(t_0)-\varphi(t)}{t_0-t}=\varphi'(s)\leq \varphi'(t_0),
\]
which implies that
\begin{equation}
\label{varphi-1}
\varphi(t) \geq \varphi(t_0)+\varphi'(t_0)(t-t_0).
\end{equation}
For $0\leq t_0<t< 1$, it follows from the Lagrange mean value theorem that there is an $s\in(t_0,t)$
such that
\[
\frac{\varphi(t)-\varphi(t_0)}{t-t_0}=\varphi'(s)\geq \varphi'(t_0),
\]
which implies that
\begin{equation}
\label{varphi-2}
\varphi(t) \geq \varphi(t_0)+\varphi'(t_0)(t-t_0).
\end{equation}
From (\ref{varphi-1}) and (\ref{varphi-2}),
we have
\begin{equation}
\label{varphi-3}
\varphi(t) \geq \varphi(t_0)+\varphi'(t_0)(t-t_0),
\quad
t\in [0,1).
\end{equation}
For $z\in \mathbb{D}$, $r\in (0,d_{\mathbb{D}}(z))$
and $\zeta\in  \mathbb{T}$,
letting
$t=|\phi(z+r\zeta)|$,
\[
t_0=\int_{0}^{2\pi}|\phi(z+re^{i\tau})|\frac{d\tau}{2\pi}
\]
in the inequality (\ref{varphi-3}),
and integrating on $\mathbb{T}$,
we have
\[
\int_{\mathbb{T}} \varphi(|\phi(z+r\zeta)|)\frac{|d\zeta|}{2\pi}
\geq
\varphi\left( \int_{\mathbb{T}}|\phi(z+r\zeta)|\frac{|d\zeta|}{2\pi}\right),
\]
which, together with the subharmonicity of $|\phi|$ and the fact that
$\varphi$ is increasing, yields that
\[
\varphi(|\phi(z)|)
\leq
\varphi\left( \int_{\mathbb{T}}|\phi(z+r\zeta)|\frac{|d\zeta|}{2\pi}\right)
\leq
\int_{\mathbb{T}} \varphi(|\phi(z+r\zeta)|)\frac{|d\zeta|}{2\pi}.
\]
Consequently, $\varphi(|\phi(z)|)$ is subharmonic in $\mathbb{D}$, which implies that
 $\omega(d_{\mathbb{D}}(\phi))$ is superharmonic in $\mathbb{D}$.
The proof of this lemma is finished.
 \end{proof}

%

 The following result easily follows from Theorem \ref{thm-1.0}.

\begin{Lem}\label{L-1}
Let  $\omega$ be a fast majorant, and let $f\in\mathscr{H}(\mathbb{D})$.
Then $f\in\Lambda_{\omega}(\mathbb{D})$ if and only if there is a positive constant $M$ such that
$$\mathscr{M}_{f}(z)\leq\,M\frac{\omega(d_{\mathbb{D}}(z))}{d_{\mathbb{D}}(z)},
\quad
z\in \mathbb{D}.
$$
Moreover, there is a positive constant $M$  which is independent of $f$  such that
         $$\frac{1}{M}\|f\|_{\Lambda_{\omega}(\mathbb{D}),s}\leq
         \sup_{z\in\mathbb{D}}\left\{\mathscr{M}_{f}(z)
         \frac{d_{\mathbb{D}}(z)}{\omega(d_{\mathbb{D}}(z))}\right\}\leq\,M\|f\|_{\Lambda_{\omega}(\mathbb{D}),s}.$$
\end{Lem}

\subsection{The proof of Theorem \ref{th-3}}
We divide the proof of this theorem into seven steps.

\noindent $\mathbf{Step~1.}$ ``$(\mathscr{F}_{2})\Rightarrow(\mathscr{F}_{1})$".

For $f\in\Lambda_{\omega_{1}}(\mathbb{D})\cap\mathscr{H}(\mathbb{D})$, it follows from
 Lemma \ref{L-1} that
there is a positive constant $M$ such that, for $w\in \mathbb{D}$,
$$\mathscr{M}_{f}(w)\leq\,M\frac{\omega_{1}(d_{\mathbb{D}}(w))}{d_{\mathbb{D}}(w)},
$$ which, together with the assumption   $(\mathscr{F}_{2})$, implies that

\beq\label{eq-0.6k}\nonumber
\int_{0}^{2\pi}\left(\mathscr{M}_{C_{\phi}(f)}(ze^{i\theta})\right)^{p}d\theta&=&
\int_{0}^{2\pi}\left(\mathscr{M}_{f}(\phi(ze^{i\theta}))|\phi'(ze^{i\theta})|\right)^{p}d\theta\\ \nonumber
&\leq&\,M^{p}\int_{0}^{2\pi}\left(\frac{\omega_{1}(d_{\mathbb{D}}(\phi(ze^{i\theta})))}{d_{\mathbb{D}}(\phi(ze^{i\theta}))}
|\phi'(ze^{i\theta})|\right)^{p}d\theta\\
&\leq&M^{2p}\left(\frac{\omega_{2}\big(d_{\mathbb{D}}(z)\big)}{d_{\mathbb{D}}(z)}\right)^{p}.
\eeq
By (\ref{eq-0.6k}) and Theorem \ref{thm-1.0}, we have $C_{\phi}(f)\in\Lambda_{\omega_{2},p}(\mathbb{D})\cap\mathscr{H}(\mathbb{D})$.\\

\noindent $\mathbf{Step~2.}$ ``$(\mathscr{F}_{1})\Rightarrow(\mathscr{F}_{2})$".

Suppose that
$C_{\phi}(f)\in\Lambda_{\omega_{2},p}(\mathbb{D})\cap\mathscr{H}(\mathbb{D})$
for every
$f\in\Lambda_{\omega_{1}}(\mathbb{D})\cap\mathscr{H}(\mathbb{D})$.
We split the proof of this step into two cases.

\noindent $\mathbf{Case~1.}$ Suppose that $$\lim_{t\rightarrow0^{+}}\frac{\omega_{1}(t)}{t}=\infty.$$

For this case, let $$\psi(t)=\frac{\omega_{1}(1-t)}{1-t}$$ for $t\in[0,1)$.
Since for $s\in (0,1]$, $$
\psi\left(1-\frac{s}{2}\right)=\frac{\omega_1(\frac{s}{2})}{\frac{s}{2}}\leq 2\frac{\omega_1(s)}{s}=2\psi(1-s),
$$
we see that $\psi$ is a doubling.
Then, by Lemma \ref{lem-0.2},
there are  holomorphic functions
$f_{j}$ $(j\in\{1,2\})$
with
\[
\sup_{\xi\in \mathbb{D}}\frac{|f_j'(\xi)|}{\psi(|\xi|)}<\infty
\]
such that
for $\xi\in\mathbb{D}$,
\[
\sum_{j=1}^{2}|f_{j}'(\xi)|\geq\psi(|\xi|),
\]
which yields that
\be\label{eq-0.8k}
|f_{1}'(\phi(z))|+|f_{2}'(\phi(z))|\geq\,\frac{\omega_{1}(d_{\mathbb{D}}(\phi(z)))}{d_{\mathbb{D}}(\phi(z))},
\quad
z\in \mathbb{D}.
\ee
It follows from Lemma \ref{L-1} that $f_{1},~f_{2}\in\Lambda_{\omega_{1}}(\mathbb{D})\cap\mathscr{H}(\mathbb{D})$.
Consequently, $f=f_{1}+\overline{f_{2}}\in\Lambda_{\omega_{1}}(\mathbb{D})\cap\mathscr{H}(\mathbb{D})$.
Since $C_{\phi}(f)\in\Lambda_{\omega_{2},p}(\mathbb{D})\cap\mathscr{H}(\mathbb{D})$, by Theorem \ref{thm-1.0},
we conclude that there is a positive constant $M$ such that, for $z\in\mathbb{D}$,
\be\label{eq-0.9k}\mathscr{J}_{9}
\leq
M\frac{\omega_{2}\big(d_{\mathbb{D}}(z)\big)}{d_{\mathbb{D}}(z)},
\ee where $$\mathscr{J}_{9}=\left(\int_{0}^{2\pi}\left(|f'_{1}(\phi(ze^{i\theta}))|+|f'_{2}(\phi(ze^{i\theta}))|\right)^{p}|\phi'(ze^{i\theta})|^{p}d\theta\right)^{\frac{1}{p}}.$$
Then $(\mathscr{F}_{2})$ follows from (\ref{eq-0.8k}) and (\ref{eq-0.9k}).

\noindent $\mathbf{Case~2.}$ Suppose that $$\lim_{t\rightarrow0^{+}}\frac{\omega_{1}(t)}{t}<\infty.$$

For this case, let $f_{0}(\xi)=\xi$ for $\xi\in\mathbb{D}$. Since $\omega_1(t)/t$ is non-increasing for $t>0$, we see that
there is a positive constant $M$ such that, for $z\in\mathbb{D}$,
\be\label{eq-10k}
|f_{0}'(\phi(z))|\geq\, M\frac{\omega_{1}(d_{\mathbb{D}}(\phi(z)))}{d_{\mathbb{D}}(\phi(z))}.
\ee
On the other hand, for $t\in(0,2]$, $$\frac{\omega_{1}(t)}{t}\geq \frac{\omega_{1}(2)}{2},$$
which gives that $f_{0}\in\Lambda_{\omega_{1}}(\mathbb{D})\cap\mathscr{H}(\mathbb{D})$.
By Theorem \ref{thm-1.0},
 there is a positive constant $M$ such that, for $z\in\mathbb{D}$,
\beqq
\left(\int_{0}^{2\pi}\left(|f'_{0}(\phi(ze^{i\theta}))||\phi'(ze^{i\theta})|\right)^{p}d\theta\right)^{\frac{1}{p}}\leq
M\frac{\omega_{2}\big(d_{\mathbb{D}}(z)\big)}{d_{\mathbb{D}}(z)},
\eeqq
which, together with (\ref{eq-10k}), implies that $(\mathscr{F}_{2})$.\\

\noindent $\mathbf{Step~3.}$ ``$(\mathscr{F}_{1})\Rightarrow(\mathscr{F}_{3})$".

By the Step 2, we see that there is a positive constant $M$ such that, for $z\in\mathbb{D}$,
\be\label{eq-1.0k}\left(\int_{0}^{2\pi}\left(|\phi'(ze^{i\theta})|
\frac{\omega_{1}\big(d_{\mathbb{D}}(\phi(ze^{i\theta}))\big)}{d_{\mathbb{D}}(\phi(ze^{i\theta}))}\right)^{p}d\theta\right)^{\frac{1}{p}}
\leq M\frac{\omega_{2}\big(d_{\mathbb{D}}(z)\big)}{d_{\mathbb{D}}(z)}.\ee
It follows from $\omega_{1}\in\mathscr{S}$ and (\ref{eq-t4}) that, for $z\in\mathbb{D}$,
\be\label{L-9}
\omega_{1}'(d_{\mathbb{D}}(\phi(z)))\leq\frac{\omega_{1}(d_{\mathbb{D}}(\phi(z)))-\omega_{1}(0)}{d_{\mathbb{D}}(\phi(z))-0}.\ee
By (\ref{eq-t6}) and (\ref{L-9}), we have
$$\mathscr{M}_{\omega_{1}(d_{\mathbb{D}}(\phi))}(z)=|\phi'(z)|\omega_{1}'
(d_{\mathbb{D}}(\phi(z)))\leq|\phi'(z)|\frac{\omega_{1}(d_{\mathbb{D}}(\phi(z)))}{d_{\mathbb{D}}(\phi(z))},$$
which, together with (\ref{eq-1.0k}), yields that
there is a positive constant $M$ such that, for $z\in\mathbb{D}$,
\beqq
\mathscr{J}_{10}
&\leq&
\left(\int_{0}^{2\pi}\left(|\phi'(ze^{i\theta})|
\frac{\omega_{1}\big(d_{\mathbb{D}}(\phi(ze^{i\theta}))\big)}{d_{\mathbb{D}}(\phi(ze^{i\theta}))}\right)^{p}d\theta\right)^{\frac{1}{p}}\\
&\leq& M\frac{\omega_{2}\big(d_{\mathbb{D}}(z)\big)}{d_{\mathbb{D}}(z)},
\eeqq where $$\mathscr{J}_{10}=\left(\int_{0}^{2\pi}\left(\mathscr{M}_{\omega_{1}(d_{\mathbb{D}}(\phi))}(ze^{i\theta})\right)^{p}d\theta\right)^{\frac{1}{p}}.$$
Then, $\omega_{1}(d_{\mathbb{D}}(\phi))\in\Lambda_{\omega_{2},p}(\mathbb{D})$
by Lemma \ref{lem-0.4}.
Consequently, $\omega_{1}(d_{\mathbb{D}}(\phi))\in\Lambda_{\omega_{2},p}(\overline{\mathbb{D}})$,
since $\phi$ is continuous in $\overline{\mathbb{D}}$.
\\

\noindent $\mathbf{Step~4.}$ ``$(\mathscr{F}_{3})\Rightarrow(\mathscr{F}_{1})$".

Since $\omega_{1}(d_{\mathbb{D}}(\phi))\in\Lambda_{\omega_{2},p}(\mathbb{D})$,
by Lemma \ref{lem-0.4}, we see that there is a positive constant $M$ such that
\be\label{eq-1.1k}
\mathscr{J}_{11}
\leq\,M\frac{\omega_{2}(d_{\mathbb{D}}(z))}{d_{\mathbb{D}}(z)},~z\in\mathbb{D},
\ee where $$\mathscr{J}_{11}=\left(\int_{0}^{2\pi}\left(|\phi'(ze^{i\theta})|\omega_{1}'(d_{\mathbb{D}}(\phi(ze^{i\theta})))\right)^{p}d\theta\right)^{\frac{1}{p}}.$$

It follows from Lemma \ref{L-1} that there is a positive constant $M$ such that, for $f\in\Lambda_{\omega_{1}}(\mathbb{D})\cap\mathscr{H}(\mathbb{D})$,
\beqq
\mathscr{M}_{C_{\phi}(f)}(ze^{i\theta})=\mathscr{M}_{f}(\phi(ze^{i\theta}))|\phi'(ze^{i\theta})|
\leq\,M\frac{\omega_{1}(d_{\mathbb{D}}(\phi(ze^{i\theta})))}{d_{\mathbb{D}}(\phi(ze^{i\theta}))}|\phi'(ze^{i\theta})|,
\eeqq
which, together with (\ref{eq-1.1k}), gives that
\beq\label{eq-1.2k}\mathscr{J}_{12}
&=&
\int_{0}^{2\pi}\left(\mathscr{M}_{f}(\phi(ze^{i\theta}))|\phi'(ze^{i\theta})|\right)^{p}d\theta\\ \nonumber
&\leq&\,M^{p}\int_{0}^{2\pi}\left(\frac{\omega_{1}(d_{\mathbb{D}}(\phi(ze^{i\theta})))}{d_{\mathbb{D}}(\phi(ze^{i\theta}))}
|\phi'(ze^{i\theta})|\right)^{p}d\theta\\ \nonumber
&\leq&M^{p}M_{1}^{p}\int_{0}^{2\pi}\left(|\phi'(ze^{i\theta})|\omega_{1}'(d_{\mathbb{D}}(\phi(ze^{i\theta})))\right)^{p}d\theta\\  \nonumber
&\leq&M^{2p}M_{1}^{p}\left(\frac{\omega_{2}(d_{\mathbb{D}}(z))}{d_{\mathbb{D}}(z)}\right)^{p},
\eeq
where $$\mathscr{J}_{12}=\int_{0}^{2\pi}\left(\mathscr{M}_{C_{\phi}(f)}(ze^{i\theta})\right)^{p}d\theta$$ and
 $$M_{1}=\sup_{z\in \mathbb{D}}\left\{\frac{\omega_{1}(d_{\mathbb{D}}(\phi(z)))}
 {d_{\mathbb{D}}(\phi(z))\omega_{1}'(d_{\mathbb{D}}(\phi(z)))}\right\}<\infty.$$
Hence $(\mathscr{F}_{1})$ follows from (\ref{eq-1.2k}) and Theorem \ref{thm-1.0}.\\

\noindent $\mathbf{Step~5.}$ ``$(\mathscr{F}_{3})\Rightarrow(\mathscr{F}_{4})$" is obvious.\\






\noindent $\mathbf{Step~6.}$ ``$(\mathscr{F}_{4})\Rightarrow(\mathscr{F}_{5})$".\\

For $r\in(0,1)$ and $\theta\in[0,2\pi]$, let
$$\mathscr{K}_{1}(re^{i\theta})=\left(\omega_{1}(d_{\mathbb{D}}(\phi(re^{i\theta})))-P[\omega_{1}(d_{\mathbb{D}}(\phi))](re^{i\theta})\right),$$
$$\mathscr{K}_{2}(re^{i\theta})=\left\{\omega_{1}(d_{\mathbb{D}}(\phi(re^{i\theta})))-\omega_{1}(d_{\mathbb{D}}(\phi(e^{i\theta})))\right\}_{+}$$
and
$$\mathscr{K}_{3}(re^{i\theta})=\left|\omega_{1}(d_{\mathbb{D}}(\phi(e^{i\theta})))-P[\omega_{1}(d_{\mathbb{D}}(\phi))](re^{i\theta})\right|.$$
Then, by $(\mathscr{F}_{4})$, we see that there is a positive constant $M$ such that
\be\label{eq-1.5k}
\left(\int_{0}^{2\pi}
\left(\mathscr{K}_{2}(re^{i\theta})\right)^{p}d\theta\right)^{\frac{1}{p}}\leq\, M\omega_{2}(d_{\mathbb{D}}(r)).
\ee
Since  $\omega_{1}(d_{\mathbb{D}}(\phi))\in\Lambda_{\omega_{2},p}(\mathbb{T})$
and $\omega_2$ is a regular majorant, by Theorem \ref{thm-5.0}, there is a positive constant $M$ such that
\be\label{eq-1.6k}
\left(\int_{0}^{2\pi}
\left(\mathscr{K}_{3}(re^{i\theta})\right)^{p}d\theta\right)^{\frac{1}{p}}\leq\, M\omega_{2}(d_{\mathbb{D}}(r)).
\ee

On the other hand, 
we have
\beqq
\mathscr{K}_{1}(re^{i\theta})&=&\left\{\omega_{1}(d_{\mathbb{D}}(\phi(re^{i\theta})))-P[\omega_{1}(d_{\mathbb{D}}(\phi))](re^{i\theta})\right\}_{+}\\
&\leq&\mathscr{K}_{2}(re^{i\theta})+\left\{\omega_{1}(d_{\mathbb{D}}(\phi(e^{i\theta})))-P[\omega_{1}(d_{\mathbb{D}}(\phi))](re^{i\theta})\right\}_{+}\\
&\leq&
\mathscr{K}_{2}(re^{i\theta})+\mathscr{K}_{3}(re^{i\theta}),
\eeqq
which, together with (\ref{eq-1.5k}), (\ref{eq-1.6k}) and the Minkowski inequality, implies that
there is a positive constant $M$ such that
\beqq
\left(\int_{0}^{2\pi}\left(\mathscr{K}_{1}(re^{i\theta})\right)^{p}d\theta\right)^{\frac{1}{p}}&\leq&
\left(\int_{0}^{2\pi}
\left(\mathscr{K}_{2}(re^{i\theta})\right)^{p}d\theta\right)^{\frac{1}{p}}\\
&&+\left(\int_{0}^{2\pi}
\left(\mathscr{K}_{3}(re^{i\theta})\right)^{p}d\theta\right)^{\frac{1}{p}}\\
&\leq&M\omega_{2}(d_{\mathbb{D}}(r)).
\eeqq

\noindent $\mathbf{Step~7.}$ ``$(\mathscr{F}_{5})\Rightarrow(\mathscr{F}_{1})$".\\

For $z,~w\in\mathbb{D}$ and $\theta\in[0,2\pi]$, let
$$\mathscr{K}_{4}(ze^{i\theta},we^{i\theta})=\left\{\omega_{1}(d_{\mathbb{D}}(\phi(ze^{i\theta})))-P[\omega_{1}(d_{\mathbb{D}}(\phi))](we^{i\theta})\right\}_{+},$$
$$\mathscr{K}_{5}(ze^{i\theta},we^{i\theta})=\left\{\omega_{1}(d_{\mathbb{D}}(\phi(ze^{i\theta})))-\omega_{1}(d_{\mathbb{D}}(\phi(we^{i\theta})))\right\}_{+}$$
and
$$\mathscr{K}_{6}(ze^{i\theta},we^{i\theta})=\left\{P[\omega_{1}(d_{\mathbb{D}}(\phi))](ze^{i\theta})-P[\omega_{1}(d_{\mathbb{D}}(\phi))](we^{i\theta})\right\}_{+}.$$
Since $\omega_{1}(d_{\mathbb{D}}(\phi))$ is superharmonic in $\mathbb{D}$ (see Lemma \ref{le-9}),
we see that

\beq\label{eq-1.6.0k}\nonumber
\mathscr{K}_{5}(ze^{i\theta},we^{i\theta})&\leq&\mathscr{K}_{4}(ze^{i\theta},we^{i\theta})\\
&\leq&\mathscr{K}_{4}(ze^{i\theta},ze^{i\theta})+\mathscr{K}_{6}(ze^{i\theta},we^{i\theta}).
\eeq
From $(\mathscr{F}_{5})$,  we see that there is a positive constant $M$
such that
\be\label{eq-1.7k}\left(\int_{0}^{2\pi}\left(\mathscr{K}_{4}(ze^{i\theta},ze^{i\theta})\right)^{p}d\theta\right)^{\frac{1}{p}}
\leq\,M\omega_{2}(d_{\mathbb{D}}(z)).\ee
Since $\omega_{1}(d_{\mathbb{D}}(\phi))\in\Lambda_{\omega_{2},p}(\mathbb{T})$
and $\omega_2$ is a regular majorant, by Theorem \ref{thm-5.0},
there is a positive constant $M$
such that
\be\label{eq-1.8k}\left(\int_{0}^{2\pi}\left(\mathscr{K}_{6}(ze^{i\theta},we^{i\theta})\right)^{p}d\theta\right)^{\frac{1}{p}}
\leq\,M\omega_{2}(|z-w|).\ee
For $z\in\mathbb{D}$ and $|z-w|\leq d_{\mathbb{D}}(z)$, it follows from (\ref{eq-1.6.0k}), (\ref{eq-1.7k}), (\ref{eq-1.8k}) and the Minkowski inequality that
\beq\label{eq-1.9k}\nonumber
\left(\int_{0}^{2\pi}\left(\mathscr{K}_{5}(ze^{i\theta},we^{i\theta})\right)^{p}d\theta\right)^{\frac{1}{p}}&\leq&
\left(\int_{0}^{2\pi}\left(\mathscr{K}_{4}(ze^{i\theta},ze^{i\theta})\right)^{p}d\theta\right)^{\frac{1}{p}}\\ \nonumber
&&+\left(\int_{0}^{2\pi}\left(\mathscr{K}_{6}(ze^{i\theta},we^{i\theta})\right)^{p}d\theta\right)^{\frac{1}{p}}\\
&\leq&M\omega_{2}(d_{\mathbb{D}}(z)).
\eeq
Combining (\ref{eq-1.9k}), Lemmas \ref{lem-0.4} and \ref{cor-1.0} gives $\omega_{1}(d_{\mathbb{D}}(\phi))\in\Lambda_{\omega_{2},p}(\mathbb{D}),$
which, together with Step 4, implies that $(\mathscr{F}_{1})$ holds. The proof of this theorem is complete.
\qed

\section*{Statements and Declarations}

\subsection*{Competing interests}
There are no competing interests.

\subsection*{Data availability}
Data sharing not applicable to this article as no datasets were generated or analysed during the current study.



\bigskip
\section{Acknowledgments}
The research of the first author was partly supported by the National Science
Foundation of China (grant no. 12071116), the Hunan Provincial
Natural Science Foundation of China (No. 2022JJ10001), the Key
Projects of Hunan Provincial Department of Education (grant no.
21A0429),
 the Double First-Class University Project of Hunan Province
(Xiangjiaotong [2018]469),  the Science and Technology Plan Project
of Hunan Province (2016TP1020),  and the Discipline Special Research
Projects of Hengyang Normal University (XKZX21002); The research of
the second author was partly supported by JSPS KAKENHI Grant Number
JP22K03363.

\end{document}